\documentclass[reqno, qed]{amsart}

\usepackage{amssymb, amsmath}
\usepackage{mathrsfs}
\usepackage{amscd}
\usepackage{verbatim}
\usepackage{enumerate}
\usepackage{color}
\usepackage[normalem]{ulem}
\usepackage{cancel}

\allowdisplaybreaks

\theoremstyle{plain}
\newtheorem{thrm}{Theorem}[section]
\theoremstyle{remark}

\newtheorem{remark}[thrm]{Remark}
\newtheorem{example}[thrm]{Example}
\theoremstyle{plain}
\newtheorem{theorem}[thrm]{Theorem}

\newtheorem{lemma}[thrm]{Lemma}
\newtheorem{proposition}[thrm]{Proposition}

\newtheorem{definition}[thrm]{Definition}

\numberwithin{equation}{section}

\newcommand{\leaderfill}{\leaders\hbox to 1em{\hss.\hss}\hfill}
\newcommand{\RR}{\mathbb{R}}
\newcommand{\R}{\mathbb{R}}
\newcommand{\CC}{\mathbb{C}}

\newcommand{\NN}{\mathbb{N}}

\newcommand{\PP}{\mathbb{P}}
\renewcommand{\P}{{\mathbb P}}

\newcommand{\E}{\mathbb{E}}

\newcommand{\e}{\varepsilon}
\newcommand{\one}{\mathbf{1}}

\newcommand{\calL}{{\mathscr L}}
\newcommand{\OO}{\Omega}

\newcommand{\F}{{\mathscr F}}

\newcommand{\g}{\gamma}
\newcommand{\lb}{\langle}
\newcommand{\rb}{\rangle}

\newcommand{\umd}{\textsc{umd}}
\newcommand{\nn}{|\!|\!|}

\newcommand{\AT}{AT}

\begin{document}

\title[Stochastic evolution equations with adapted drift]{A new approach to stochastic evolution equations with adapted drift}



\author{Matthijs Pronk and Mark Veraar}
\address{Delft Institute of Applied Mathematics\\
Delft University of Technology \\ 2600 GA Delft\\The
Netherlands} \email{Matthijs.Pronk@tudelft.nl, matthijs.pronk@inphykem.com}
\address{Delft Institute of Applied Mathematics\\
Delft University of Technology \\ 2600 GA Delft\\The
Netherlands} \email{M.C.Veraar@tudelft.nl}

\begin{abstract}
In this paper we develop a new approach to stochastic evolution equations with an unbounded drift $A$ which is dependent on time and the underlying probability space in an adapted way. It is well-known that the semigroup approach to equations with random drift leads to adaptedness problems for the stochastic convolution term. In this paper we give a new representation formula for the stochastic convolution which avoids integration of nonadapted processes. Here we mainly consider the parabolic setting. We establish connections with other solution concepts such as weak solutions. The usual parabolic regularity properties are derived and we show that the new approach can be applied in the study of semilinear problems with random drift. At the end of the paper the results are illustrated with two examples of stochastic heat equations with random drift.
\end{abstract}

\thanks{The first named author is supported by VICI subsidy 639.033.604
of the Netherlands Organisation for Scientific Research (NWO)}

\maketitle

\section{Introduction}

Let $E_0$ be a Banach space and let $H$ be a separable Hilbert space. Let $(\OO, \F, \PP)$ be a complete probability space with a filtration $(\F_t)_{t\in[0,T]}$. We study the following stochastic evolution equation on $E_0$.
\begin{align}\label{problem:SEE-ReprForm}
\left\{
  \begin{array}{ll}
    dU(t) &= (A(t)U(t) + F(t,U(t)))\;dt + B(t,U(t))\;dW(t), \\
    U(0) &= u_0.
  \end{array}
\right.
\end{align}
Here $(A(t,\omega))_{t\in [0,T], \omega\in\OO}$ is a measurable and adapted family of unbounded operators on $E_0$. Moreover, $F$ and $B$ are semilinear nonlinearities and $W$ is a cylindrical Brownian motion.

The integrated form of \eqref{problem:SEE-ReprForm} often leads to problems as in general $A(t) U(t)$ is not well-defined or not integrable with respect to time. In the semigroup approach to \eqref{problem:SEE-ReprForm} this difficulty does not occur. We refer to the monograph \cite{DPZ} and references therein for details on the semigroup approach to \eqref{problem:SEE-ReprForm} in the Hilbert space setting. Extensions to the class of Banach spaces with martingale type $2$ can be found in \cite{Brz97} in the case $A$ is not depending on time. An extension to the nonautonomous setting (i.e. $A$ depends on time) can be found in \cite{Veraar-SEE}.
In the time-dependent setting the mild formulation of \eqref{problem:SEE-ReprForm} is usually written in the following form:
\begin{equation}\label{eq:mildintro}
U(t) = S(t,0) u_0 + \int_0^t S(t,s) F(s,U(s))\, ds + \underbrace{\int_0^t S(t,s) B(s,U(s))\, d W(s)}_{\text{well-defined?}}
\end{equation}
Here, $(S(t,s,\omega))_{0\leq s\leq t\leq T, \omega\in\OO}$ is the evolution system generated by $(A(t,\omega))_{t\in [0,T]}$.
In this case, there is an obstruction in the mild formulation of a solution. The problem is that $\omega\mapsto S(t,s,\omega)$ does not satisfy the right adaptedness properties. In general $\omega\mapsto S(t,s,\omega)$ is only $\F_t$-measurable and not $\F_s$-measurable (see Example \ref{ex:noadapted}). Therefore, the stochastic integral in \eqref{eq:mildintro} cannot be defined in the sense of It\^o.
Equations with random generators arise naturally in the case $A$ depends on a stochastic process, e.g.\ in filtering theory (see \cite{Xiong} and references therein). A random drift term of the form $A(t,\omega) = A_0 + A_1(t,\omega)$, where $A_0$ is a fixed differential operator and $A_1(t,\omega)$ a lower order perturbation of $A_0$, can be treated by a perturbation argument. This case is standard and easy to deal with and we will not consider it here. Our model case is the situation, where the highest order coefficients are also dependent on $(t,\omega)$.

There are several known approaches to \eqref{problem:SEE-ReprForm} which allow a random drift term. In the method of monotone operators (see \cite{KR79}, \cite{Par2}, \cite{Prevot-Roeckner}, \cite{Rozov}) the problem \eqref{problem:SEE-ReprForm} is formulated on a Hilbert space and one can use Galerkin approximation well-posedness questions to reduce the problem to the finite-dimensional setting. In this way no additional difficulty arises when $A$ is dependent on $\Omega$ and time. Also in the $L^p$-approach of Krylov \cite{Kry} one can allow the coefficient of a second order operator $A$ on $\R^d$ to be dependent on $\Omega$ and time in a measurable way. The above mentioned approaches do not use the mild formulation \eqref{eq:mildintro}.

Mild formulations can be useful in many type of fixed point arguments. They can also used to study long time behavior (invariant measures) and time regularity. There have been several attempts to extend the mild approach to \eqref{problem:SEE-ReprForm} to the $\omega$-dependent setting. A possible method for \eqref{problem:SEE-ReprForm} using mild formulations is to use stochastic integration for nonadapted integrands and Malliavin calculus. This has been studied in \cite{AlLeNu, AlNuVi, NualartLeon, LeNu-AnticipIntegrEqs, NuaVie}. This approach is based on Skorohod integration techniques and it requires certain Malliavin differentiability of the operators $A(t)$ or $S(t,s)$.
Another basic example where non-adapted integrands naturally occur is when the initial value of an SDE or SPDE depends on the full paths of the underlying stochastic process (see \cite{Buckdahn91, MilletNualartSanz, Olivera}).
Finally, we mention that in \cite{NVW11eq} a maximal regularity approach to \eqref{problem:SEE-ReprForm} with random $A$ has been developed.

In this paper we will develop a new method for the stochastic evolution equation \eqref{problem:SEE-ReprForm} with random $A$. It is based on a new representation formula for stochastic convolutions. In order to explain this representation formula, consider
\begin{equation}\label{eq:SEEsimpleintro}
\left\{
    \begin{array}{ll}
        dU(t) &= A(t)U(t)\;dt + G\;dW(t), \\
        U(0) &= 0,
    \end{array}
\right.
\end{equation}
where $G$  is an adapted and measurable process and $A$ is as before.
Our new representation formula for the solution to \eqref{eq:SEEsimpleintro} is:
\begin{align}\label{eq:reprformintro}
U(t) = - \int_0^t S(t,s) A(s) I(\one_{(s,t)}G) \, ds + S(t,0) I(\one_{(0,t)} G),
\end{align}
where $I(\one_{(s,t)}G) = \int_s^t G \, d W$. This will be called the {\em pathwise mild solution} to \eqref{eq:SEEsimpleintro}. The pathwise mild solution \eqref{eq:reprformintro} can basically be obtained by using integration by parts formula for the stochastic convolution. The advantage of the formulation is that it does not require stochastic integration of nonadapted integrands. A difficulty in \eqref{eq:reprformintro} is that the norm of the operator-valued kernel $S(t,s) A(s)$ is usually of order $(t-s)^{-1}$. Fortunately, the Bochner integral in \eqref{eq:reprformintro} can still be shown to be convergent as the paths of the process $t\mapsto I(\one_{(0,t)}G)$ have additional H\"older or Sobolev regularity.
The pathwise mild solution will be shown to be equivalent to weak, variational and forward mild solutions (see Section \ref{sec:repr}).

In order to have evolution families with sufficient regularity properties, we will restrict ourselves to the parabolic setting. We will assume that the operators $(A(t))_{t\in [0,T]}$ satisfy the so-called (AT)-conditions which were introduced by Acquistapace and Terreni. This is a combination of a uniform sectoriality condition and a H\"older condition on the resolvents. We will allow $\omega$-dependent H\"older constants in the latter, which is important in view of applications. However, we would like to note that the pathwise mild solution \eqref{eq:reprformintro} can also be used in other parabolic situations where $\|S(t,s,\omega) A(s,\omega)\|\leq C(\omega) (t-s)^{-1}$, or even in a general setting if we assume $G$ is regular in space (i.e.\ takes values in a suitable subspace of the domains of $A$).
The pathwise mild solution \eqref{eq:reprformintro} is a pathwise integral representation of the solution and we expect that this has potential applications in the theory of stochastic evolution equations even in the case of non-random $A$. For instance, as there is no stochastic convolution, certain behavior can be studied path-wise. This could play an important role in numerical simulations, in studying long-term behavior and it could be combined with methods from rough path theory. Moreover, when replacing $W$ by a general Gaussian process $M$ (or more general stochastic process), it is enough to establish an integration theory for $I(\one_{(s,t)}G) = \int_s^t G \, d M$, and no stochastic convolutions are needed in \eqref{eq:reprformintro}.

This paper is organized as follows. In Section \ref{eq:SEF} we will discuss the (AT)-conditions, and extend some of their results to the $\omega$-dependent setting. In Section \ref{sec:pathwisereg} we present a new pathwise regularity result, which will allow to obtain the usual parabolic regularity of the solution to \eqref{eq:SEEsimpleintro}. In Section \ref{sec:repr} we discuss the pathwise mild solution \eqref{eq:reprformintro} and its relations to other solution concepts. In Section \ref{sec:semil} we discuss a general semilinear problem and prove well-posedness with a fixed point argument. For this we first obtain  well-posedness under the assumption that the constants in the (AT)-conditions are $\omega$-independent H\"older conditions. After that we localize the H\"older condition and extend the result to the general case. Finally, we illustrate our results with examples in Section \ref{sec:ex}.

\medskip

{\em Acknowledgments.} The authors thank the anonymous referee for careful reading and helpful suggestions.

\section{Stochastic evolutions families\label{eq:SEF}}

Let $E_0$ be a Banach space. In this section we will be concerned with
generation properties of families of unbounded operators.
For $t\in [0,T]$ and $\omega\in \OO$ fixed, we consider a closed and densely defined operator
\[A(t,\omega): E_0 \supset D(A(t,\omega)) \to E_0.\]
For convenience, we sometimes write $A(t)$ and $D(A(t))$ instead of $A(t,\omega)$ and $D(A(t,\omega))$, respectively.

We will only consider the parabolic setting (i.e.\ the case the operators $-A(t,\omega)$ are sectorial of angle less than $\pi/2$ with uniform constants in $(t,\omega)$. This is well-documented in the literature (see \cite{Ama, Lunardi, Pazy, Tanabe, Ta2}).

\subsection{Generation theorem}\label{subsec:gen-thm}
In this subsection we will consider the conditions introduced by Acquistapace and Terreni \cite{AT2} (see also \cite{Acquist:evoloper, AT1, Ama, Schnaubelt, Ta2, Yag90, Yag91} and references therein).
An important difficulty in our situation is that $A(t,\omega)$ depends on the additional parameter $\omega\in \OO$.

For $\vartheta \in (\pi/2,\pi)$ we set
\[\Sigma_{\vartheta} = \{\lambda\in \CC:\ |\arg \lambda| < \vartheta\}.\]

On $A$ we will assume a sectoriality condition and a H\"older continuity assumption:

\let\ALTERWERTA\theenumi
\let\ALTERWERTB\labelenumi
\def\theenumi{(AT1)}
\def\labelenumi{(AT1)}
\begin{enumerate}
  \item\label{asmp:AT1} There exists a $\vartheta\in (\pi/2,\pi)$ and $M>0$ such that for every $(t,\omega)\in [0,T]\times\Omega$, one has $\Sigma_\vartheta \cup \{0\}\subset \rho(A(t,\omega))$ and
  \[\|R(\lambda, A(t,\omega))\|_{\calL(E_0)}\leq \frac{M}{|\lambda|+1}, \ \ \lambda \in \Sigma_\vartheta \cup \{0\}.\]
\end{enumerate}
\let\theenumi\ALTERWERTA
\let\labelenumi\ALTERWERTB

\let\ALTERWERTA\theenumi
\let\ALTERWERTB\labelenumi
\def\theenumi{(AT2)}
\def\labelenumi{(AT2)}
\begin{enumerate}
  \item\label{asmp:AT2} There exist $0<\nu,\mu\leq 1$ with $\mu+\nu>1$ such that for every $\omega\in \OO$, there exists a constant $L(\omega) \geq 0$ such that for all $s,t\in [0,T]$ and $\lambda \in \Sigma_{\vartheta}$,
    \begin{align*}
    |\lambda|^{\nu}\|A(t,\omega)&R(\lambda,A(t,\omega))(A(t,\omega)^{-1}-A(s,\omega)^{-1})\|_{\calL(E_0)} \leq L(\omega)|t-s|^{\mu}.
    \end{align*}
\end{enumerate}
\let\theenumi\ALTERWERTA
\let\labelenumi\ALTERWERTB

We would like to point out that it will be important that in the H\"older continuity assumption the H\"older constant is allowed to depend on $\omega$.
Whenever \ref{asmp:AT1} and \ref{asmp:AT2} hold, it is said that (\AT) holds. The abbreviation (\AT) stands for Acquistapace and Terreni.

In the sequel we will not write the dependence on $\omega\in \OO$ explicitly whenever there is no danger of confusion.

\begin{example}\label{ex:Tanabe}
Assume $E_1 = D(A(t,\omega))$ is constant with uniform estimates in $t\in [0,T]$ and $\omega\in \OO$. Assume \ref{asmp:AT1} holds. If there is a $\mu\in (0,1]$ and a mapping $C:\OO\to \R_+$ such that
\[\|A(t) - A(s)\|_{\calL(E_1, E_0)} \leq C|t-s|^{\mu}, \ \ s,t\in [0,T], \]
then \ref{asmp:AT2} holds with $\nu = 1$ and $L = M C$ up to a constant multiplicative factor.
The above type of condition is sometimes called the Kato--Tanabe condition (see \cite{Pazy, Tanabe}).
\end{example}

Let $\Delta := \{(s,t) \in [0,T]^2:\ s\leq t\}$. The following result can be derived by applying \cite[Theorem 2.3]{Acquist:evoloper} pointwise in $\OO$.
\begin{theorem}\label{thrm:exist-evol}
Assume (\AT). There exists a unique map $S:\Delta\times\OO \to \calL(E_0)$ such that
\begin{enumerate}
  \item For all $t\in [0,T]$, $S(t,t) = I$.
  \item For $r\leq s\leq t$, $S(t,s)S(s,r) = S(t,r)$.
  \item For every $\omega\in \OO$, the map $S(\cdot,\omega)$ is strongly continuous.
  \item There exists a mapping $C:\OO\to \R_+$ such that for all $s\leq t$, one has $\|S(t,s)\|\leq C$.
  \item For every $s<t$, one has $\frac{d}{dt} S(t,s) = A(t)S(t,s)$ pointwise in $\OO$, and there exists a mapping $C:\OO\to \R_+$ such that
      \[\|A(t)S(t,s)\|_{\calL(E_0)} \leq C(t-s)^{-1}.\]
\end{enumerate}
\end{theorem}

In the above situation we say that $(A(t))_{t\in [0,T]}$ {\em generates the evolution system/family} $(S(t,s))_{0\leq s\leq t\leq T}$.

\subsection{Measurability}
Throughout this subsection we assume that (\AT) holds.

As the domains $D(A(t,\omega))$ also vary in $(t,\omega)$, the most natural way to formulate the adaptedness assumption is by a condition on the resolvent as follows:

\let\ALTERWERTA\theenumi
\let\ALTERWERTB\labelenumi
\def\theenumi{(H1)}
\def\labelenumi{(H1)}
\begin{enumerate}
  \item\label{asmp:A-mble-adptd} For some $\lambda\in\Sigma_\vartheta\cup \{0\}$, $R(\lambda, A(\cdot)): [0,T] \times \OO \to \calL(E_0)$ is strongly measurable and adapted.
\end{enumerate}
\let\theenumi\ALTERWERTA
\let\labelenumi\ALTERWERTB

Here we consider measurability and adaptedness in the uniform operator topology. Hypothesis \ref{asmp:A-mble-adptd} implies that for all $\lambda\in\Sigma_\vartheta\cup \{0\}$, $R(\lambda, A(\cdot))$ is strongly measurable and adapted. This follows from the fact that the resolvent can be expressed as a uniformly convergent power series (see \cite[Proposition IV.1.3]{EngelNagel}).

\begin{example}\label{example:ConstDomH1Verify}
Assume the conditions of Example \ref{ex:Tanabe} hold. If $A:[0,T]\times\OO\to \calL(E_1, E_0)$ is strongly measurable and adapted, then \ref{asmp:A-mble-adptd} holds. Indeed, fix $\omega_0\in \OO$. Since $(t,\omega)\mapsto A(t,\omega) A(0,\omega_0)^{-1}$ is strongly measurable and adapted and taking inverses is continuous on the open set of invertible operators, it follows that $(t,\omega)\mapsto A(0,\omega_0) A(t,\omega)^{-1}$ is strongly measurable and adapted. This clearly yields \ref{asmp:A-mble-adptd}.
\end{example}

Let $r>0$ and $\eta \in (\pi/2, \vartheta)$, and consider the counterclockwise oriented curve
\[\gamma_{r,\eta} := \{\lambda\in \CC:\ |\arg\lambda| = \eta,\ |\lambda| \geq r\} \cup \{\lambda \in \CC:\ |\lambda| = r, -\eta \leq \arg\lambda\leq \eta\}.\]
For $s\in [0,T]$, consider the analytic semigroup $(e^{tA(s)})_{t\geq 0}$ defined by
\begin{align*}
e^{tA(s)}x = \left\{
    \begin{array}{ll}
        \frac{1}{2\pi i} \int_{\gamma_{r,\eta}} e^{t\lambda} R(\lambda, A(s))x\;d\lambda, & t>0, \\
        x, & t=0.
    \end{array}
    \right.
\end{align*}

\begin{proposition}\label{prop:evolsys-adapted}
The evolution system $S: \Delta \times \OO \to \calL(E_0)$ is strongly measurable in the uniform operator topology. Moreover, for each $t \geq s$, $\omega\mapsto S(t,s,\omega)\in \calL(E_0)$ is strongly $\mathscr{F}_t$-measurable in the uniform operator topology.
\end{proposition}
In Example \ref{ex:noadapted} we will show that the above measurability result cannot be improved in general.
\begin{proof}
Fix $0\leq s<t\leq T$. The evolution system $S(t,s)$ is given in \cite{Acquist:evoloper}, as follows. Let $Q(t,s)$ be given by
\begin{align}
Q(t,s) = A(t)^2 e^{(t-s)A(t)} (A(t)^{-1} - A(s)^{-1}).
\end{align}
Define inductively $Q_n(t,s)$ by
\[Q_1(t,s) = Q(t,s),\quad Q_n(t,s) = \int_s^t Q_{n-1}(t,r)Q(r,s)\;dr.\]
Then the evolution system $S(t,s)$ is given by
\begin{align}
S(t,s) = e^{(t-s)A(s)} + \int_s^t Z(r,s)\;dr,
\end{align}
where
\begin{align}
\begin{split}
    Z(t,s) :&= A(t)e^{(t-s)A(t)} - A(s)e^{(t-s)A(s)}
    \\ & \quad + \sum_{n=1}^\infty \int_s^t Q_n(t,r)\Big(A(r)e^{(r-s)A(r)} - A(s)e^{(r-s)A(s)}\Big)\;dr \\
    & \quad + \sum_{n=1}^\infty \int_s^t (Q_n(t,r) - Q_n(t,s))A(s)e^{(r-s)A(s)}\;dr \\ & \quad + \sum_{n=1}^\infty Q_n(t,s)(e^{(t-s)A(s)}-1).
\end{split}
\end{align}
The above series converges in $\calL(E_0)$, see \cite[Lemma 2.2 (1)]{Acquist:evoloper}.

\textit{Step 1: $S(t,s)$ is $\mathscr{F}_t$-measurable.}
Note that
\[A(t)^ne^{(t-s)A(t)} = \frac{1}{2\pi i} \int_{\gamma_{r,\eta}} e^{(t-s)\lambda} \lambda^n R(\lambda, A(t))\;d\lambda,\qquad n\in \NN.\]
Also, for $n\in \NN$, note that $\lambda \mapsto e^{(t-s)\lambda} \lambda^n R(\lambda, A(t))$ is continuous on $\gamma_{r,\eta}$, and hence Riemann integrable. The random variable $e^{(t-s)\lambda} \lambda^n R(\lambda, A(t))$ is $\mathscr{F}_t$-measurable for every $\lambda \in \gamma_{r,\eta}$, hence every Riemann sum and thus every Riemann integral is $\mathscr{F}_t$-measurable. It follows that $A(t)^ne^{(t-s)A(t)}$ (as it is the limit of Riemann integrals) is $\mathscr{F}_t$-measurable. In particular this holds for $n=2$, and hence $Q(t,s)$ is $\mathscr{F}_t$-measurable as well.
On $(s,t)$, the map $r\mapsto Q_{n-1}(t,r)Q(r,s)$ is continuous, by \cite[Lemma 2.1]{Acquist:evoloper}. Thus, by a similar argument as above, $Q_n(t,s)$ is $\mathscr{F}_t$-measurable, for $n\geq 2$. Also $e^{(t-s)A(s)}$ is $\mathscr{F}_t$-measurable. Hence the random variable
\[\sum_{n=1}^\infty Q_n(t,s)(e^{(t-s)A(s)}-1)\]
is $\mathscr{F}_t$-measurable.

Clearly $r\mapsto A(s)e^{(r-s)A(s)}$ is continuous. By \cite[Lemma 2.1]{Acquist:evoloper}, $r\mapsto Q_n(t,r) - Q_n(t,s)$ is continuous as well. Hence, as before we see that
\[\sum_{n=1}^\infty \int_s^t (Q_n(t,r) - Q_n(t,s))A(s)e^{(r-s)A(s)}\;dr \]
is $\mathscr{F}_t$-measurable.

The map $g: r\mapsto A(r)e^{(r-s)A(r)} - A(s)e^{(r-s)A(s)}$ for $r\in (s,t)$ is continuous. Indeed,
\begin{align*}
\|g(q) - g(r)\|_{\calL(E_0)} &\leq \|A(q)e^{(q-s)A(q)} - A(r)e^{(q-r)A(r)}\|_{\calL(E_0)} \\ & \quad + \|A(r)(e^{(q-s)A(r)} - e^{(r-s)A(r)})\|_{\calL(E_0)} \\ & \quad + \|A(s)(e^{(r-s)A(s)} - e^{(q-s)A(s)})\|_{\calL(E_0)}.
\end{align*}
Now \cite[Lemma 1.10(i)]{AT3} yields the required continuity of $g$ and its integral will be $\F_t$-measurable again.
Combining all terms we deduce that $Z(t,s)$ is $\mathscr{F}_t$-measurable.
By \cite[Lemma 2.2(ii)]{Acquist:evoloper} the map $r\mapsto Z(r,s)$ is continuous on $(s,t)$ and therefore, we can now deduce that $S(t,s)$ is $\mathscr{F}_t$-measurable.

\textit{Step 2: measurability of the process $S$.}
For $n\in \NN$ and $k=0,1,\ldots, n-1$, consider the triangle
\[D_{k,n} = \{(s,t) \in [0,T]^2:\ \tfrac{k}{n} \leq t \leq \tfrac{k+1}{n},\ \tfrac{k}{n} \leq s\leq t\}.\]
Let $I$ be the identity operator on $E_0$. Then for $0\leq s\leq t\leq T$, define $X_n: \Delta \times \OO \to \calL(E_0)$ by
\[X_n(t,s) := \sum_{k=1}^{n-1} \one_{D_{k,n}}(s,t) I + \sum_{k=0}^{n-2} \sum_{m=k+1}^{n-1} \one_{(\frac{k}{n}, \frac{k+1}{n}] \times (\frac{m}{n}, \frac{m+1}{n}]}(s,t) S\Big(\frac{m}{n}, \frac{k}{n}\Big).\]

\begin{comment}
Consider $B \in \mathscr{B}(\calL(E_0))$. If $I \in B$, then
\[(X_n)^{-1}(B) = \Big( \bigcup_{k=0}^{n-1} D_{k,n}(s,t) \times \OO \Big) \cup \Big( \bigcup_{k=0}^{n-2} \bigcup_{m=k+1}^{n-1} \Big( \frac{k}{n}, \frac{k+1}{n}\Big] \times \Big(\frac{m}{n}, \frac{m+1}{n}\Big] \times S\Big(\frac{m}{n}, \frac{k}{n}\Big)^{-1}(B) \Big).\]
If $I\not\in B$, then
\[(X_n)^{-1}(B) = \Big( \bigcup_{k=0}^{n-1} D_{k,n}(s,t) \times \{0\} \Big) \cup \Big( \bigcup_{k=0}^{n-2} \bigcup_{m=k+1}^{n-1} \Big( \frac{k}{n}, \frac{k+1}{n}\Big] \times \Big(\frac{m}{n}, \frac{m+1}{n}\Big] \times S\Big(\frac{m}{n}, \frac{k}{n}\Big)^{-1}(B) \Big).\]
In both cases, it follows that $X_n: \Delta \times \OO \to \calL(E_0)$ is measurable.
\end{comment}

Since $S(t,s): \OO \to \calL(E_0)$ is strongly measurable, by Step 1, it follows that $X_n: \Delta \times \OO \to \calL(E_0)$ is strongly measurable. Moreover, by strong continuity of $S$, pointwise on $\Delta\times \OO$, one has $X_n \to S$. Hence $S: \Delta \times \OO \to \calL(E_0)$ is strongly measurable.
\end{proof}

\begin{example}
\label{ex:noadapted}
Let $E_0 = \R$ and let $A:[0,T]\times\OO\to \R$ be a measurable and adapted process such that $\sup_{t\in [0,T]}|A(t,\omega)|<\infty$.
Then $A$ generates the evolution system
\[S(t,s,\omega) = \exp\Big(\int_s^t A(r,\omega) \, dr \Big).\]
Obviously $\omega\mapsto S(t,s,\omega)$ is only $\F_t$-measurable in general.
\end{example}

\subsection{Pathwise regularity properties of evolution families}
Throughout this subsection we assume that (\AT) holds.
First we recall some facts from interpolation theory. An overview on the subject can be found in \cite{Ama,Lunardi,Tr1}.

For $\theta \in (0,1)$ and $p\in [1,\infty]$ the real interpolation space
$E_{\theta,p}^t := (E_0,D(A(t)))_{\theta,p}$ is the subspace of all $x\in E_0$ for which
\begin{align}\label{defn:intrpltn-nrm}
\|x\|_{(E_0, D(A(t)))_{\theta,p}} := \Big(\int_0^\infty s^{p(1-\theta)} \|A(t)e^{sA(t)}x\|_{E_0}^p\;\frac{ds}{s}\Big)^{1/p}<\infty,
\end{align}
with the obvious modification if $p=\infty$. Clearly, the space $E_{\theta,p}^t$ and its norm also depends on $\omega\in \Omega$, but this will be omitted from the notation. The space $E_{\theta,p}^t$ with the above norm is a Banach space. For convenience we also let $E_{0,p}^t:=(E_0,D(A(t)))_{0,p} = E_0$ and $E_{1,p}^t:=(E_0, D(A(t)))_{1,p} = D(A(t))$. By applying $A(t)$ finitely many times on both sides we extend the definition of the spaces $E_{\theta,p}^t:=(E_0, D(A(t)))_{\theta,p}$ to all $\theta\geq 0$.

For all $\theta\in [0,\alpha)$
\begin{equation}\label{array-cnt-emb1}
\begin{aligned}
E_{\alpha, 1}^t  & \hookrightarrow E_{\alpha,p}^t \hookrightarrow E_{\alpha, \infty}^t \hookrightarrow E_{\theta, 1}^t  \hookrightarrow E_0.
\end{aligned}
\end{equation}
Here, the embedding constants only depend on the constants in \ref{asmp:AT1} and thus are independent of time and $\omega$.

For $\theta\in (0,1)$, let $(-A(t,\omega))^{-\theta}$ be defined by
\begin{align*}
(-A(t))^{-\theta} = \frac{1}{\Gamma(\theta)} \int_0^\infty s^{\theta-1} e^{sA(t)}\;ds,
\end{align*}
and let $(-A(t))^{\theta} = ((-A(t))^{-\theta})^{-1}$ with as domain the range of $(-A(t))^{-\theta}$.
Endowed with the norm $\|x\|_{D((-A(t))^{\theta})} = \|(-A(t))^{\theta}x\|_{E_0}$, the space $D((-A(t))^{\theta})$ becomes a Banach space.

For $\theta\geq 0$, the following continuous embeddings hold:
\begin{align}\label{array-cnt-emb2}
E_{\theta, 1}^t  \hookrightarrow D(-A(t))^{\theta} \hookrightarrow E_{\theta, \infty}^t.
\end{align}
and again the embedding constants only depend on the constants in \ref{asmp:AT1}.

The next result follows from pointwise application of \cite[(2.13), (2.15) and Proposition 2.4]{Schnaubelt}.
Recall that $\mu,\nu \in (0,1]$ are the smoothness constants from \ref{asmp:AT2}.
\begin{lemma}\label{lem:-evolsys-frct-ineq}
There exists a mapping $C:\OO\to \R_+$ such that for all $0\leq s<t\leq T$, $0\leq \alpha<\beta \leq 1$, $\eta\in (0,\mu+\nu-1)$, $\gamma\in [0,\mu)$, $\theta\in [0,1]$ and $\delta,\lambda\in (0,1)$, the following inequalities hold
\begin{align}
\|S(t,s)x\|_{E^t_{\beta,1}} &\leq C\frac{\|x\|_{E^s_{\alpha,\infty}}}{(t-s)^{\beta - \alpha} }, &x&\in E^s_{\alpha,\infty}. \label{thrm-evolsys-frct-ineq1}
\\ \|A(t) S(t,s) x\|_{E^t_{\eta,1}} &\leq C\frac{\|x\|_{E_0}}{(t-s)^{1+\eta}},   &x& \in E_0.\label{thrm-evolsys-frct-ineq2}
\\ \|A(t) S(t,s) x\|_{E^t_{\eta,1}} &\leq C\frac{\|x\|_{E^s_{\delta,\infty}}}{(t-s)^{1+\eta-\delta}},  &x& \in E^s_{\delta,\infty}.\label{thrm-evolsys-frct-ineq2c}
\\ \|S(t,s)(-A(s))^{\gamma} x\|_{E_0} &\leq C\frac{\|x\|_{E_0}}{(t-s)^{\gamma}},  &x& \in D((-A(s))^{\gamma}), \label{thrm-evolsys-frct-ineq2b}
\\ \|(-A(t))^{\theta}S(t,s)(-A(s))^{-\theta}\|_{\calL(E_0)} &\leq C \label{thrm-evolsys-frct-ineq3}
\end{align}
and $\Delta\ni (t,s)\mapsto (-A(t))^{\theta}S(t,s)(-A(s))^{-\theta}$ is strongly continuous.
\end{lemma}
In general $C$ depends on the constants of \ref{asmp:AT1} and \ref{asmp:AT2}.
Note that to obtain \eqref{thrm-evolsys-frct-ineq1} one needs to use reiteration in order to obtain the improvement from exponent $\infty$ to $1$. Moreover, \eqref{thrm-evolsys-frct-ineq2c} follows from interpolation of \eqref{thrm-evolsys-frct-ineq1} and \eqref{thrm-evolsys-frct-ineq2} and reiteration.

\medskip

\subsection{Improved regularity under adjoint conditions}
Throughout this section we assume the (\AT)-conditions hold and that $E_0$ is reflexive.
In this section we will obtain further pathwise regularity properties by duality arguments.
Then $(A(t)^*)_{t\in [0,T],\omega\in\OO}$ is a family of closed densely define operators on $E_0^*$. Moreover, since $R(\lambda,A(t)^*) = R(\lambda,A(t))^*$, \ref{asmp:AT1} holds for this family as well. Furthermore, we will assume that the family of adjoints satisfies \ref{asmp:AT2} with constants $\mu^*$ and $\nu^*$, throughout the rest of this section.

Under the above assumption on the adjoint family, we know that for every $t\in (0,T]$, the family $(A(t-\tau)^*)_{\tau\in [0,t]}$ satisfies the (\AT)-conditions as well, and therefore by Theorem \ref{thrm:exist-evol} it generates an evolution family:
\begin{align}\label{defn:evoloper-adjoint}
(V(t;\tau,s))_{0\leq s\leq \tau\leq t}.
\end{align}
Recall from \cite[Proposition 2.9]{AT2}, that $S(t,s)^* = V(t;t-s,0)$, and by Theorem \ref{thrm:exist-evol} (5) and the chain rule, for $s<t$
\begin{align}\label{eqn:adjoint-evol-drvtv}
\frac{d}{ds} S(t,s)^* = -A(s)^*S(t,s)^*.
\end{align}
Moreover, for all $x\in D(A(t)^*) = D(A(t-0)^*)$ one has $s\mapsto S(t,s)^*x^* = V(t;t-s,0) x^*$ is continuously differentiable on $[0,t]$.

\begin{lemma}\label{lem:diff-prop-evolsys}
Assume the above conditions. For every $t\in (0,T]$, the mapping $s\mapsto S(t,s)$ belongs to $C^1([0,t);\calL(E_0))$,
and for all $x\in D(A(s))$ one has $\frac{d}{ds} S(t,s)x = -S(t,s)A(s)x$.
For all $0\leq s<t\leq T$, $\beta\in [0,1]$, $0<\gamma<\mu^*+\nu^*-1$, $0\leq \theta < \mu^*+\nu^*-1$, $\delta\in (0,1)$, the following inequalities hold:
\begin{align}
\|S(t,s)(-A(s))^{\beta} x\|_{E_0} &\leq C\frac{\|x\|_{E_0}}{(t-s)^{\beta}}, &x& \in D((-A(s))^{\beta}). \label{thrm-evolsys-frct-ineq42}
\\ \|S(t,s)(-A(s))^{1+\theta}x\|_{E_0} &\leq C\frac{\|x\|_{E_0}}{(t-s)^{1+\theta}}, &x& \in D((-A(s))^{1+\theta}). \label{thrm-evolsys-frct-ineq4}
\\
\|A(t)^{-\delta} S(t,s)(-A(s))^{1+\gamma} x\|_{E_0} &\leq C\frac{\|x\|_{E_0}}{(t-s)^{1+\gamma-\delta}},
& x& \in D((-A(s))^{1+\gamma}).
\label{thrm-evolsys-frct-ineq5}
\end{align}
\end{lemma}
In particular, we see that for every $s<t$, the operator $S(t,s)A(s)$ uniquely extends to a bounded operator on $E_0$ of norm $C(t-s)^{-1}$, which will be denoted by $S(t,s)A(s)$ again. As before, the constant $C$ depends on the constants in the (\AT)-conditions for $A$ and $A^*$.

Note that the additional condition on the adjoint family yields the improvement \eqref{thrm-evolsys-frct-ineq42} of \eqref{thrm-evolsys-frct-ineq2b}.

\begin{proof}
It follows from \eqref{eqn:adjoint-evol-drvtv} that
\[\frac{d}{ds} S(t,s) = \Big(\frac{d}{ds} S(t,s)^*\Big)^* = (-A(s)^*S(t,s)^*)^*,\]
where we identify $E_0$ and $E_0^{**}$. Using $S(t,s)^* = V(t;t-s,0)$ as in \eqref{defn:evoloper-adjoint}, it follows that $(-A(s)^*S(t,s)^*)^* \in \calL(E_0)$. Hence, for any $x\in D(A(s))$ and every $x^* \in E_0^*$, one has
\begin{align*}
\lb (-A(s)^*S(t,s)^*)^*x, x^* \rb = - \lb x, A(s)^*S(t,s)^*x^* \rb = \lb -S(t,s)A(s)x, x^*\rb.
\end{align*}
By a Hahn-Banach argument, we obtain $\frac{d}{ds}S(t,s)x = -S(t,s)A(s)x$ for all $x\in D(A(s))$.

By \eqref{thrm-evolsys-frct-ineq1} if $\theta=0$ and otherwise \eqref{thrm-evolsys-frct-ineq2} for the adjoint family, we find that
\begin{align*}
\|(-A(s)^*)^{1+\theta}S(t,s)^*\|_{\calL(E_0^*)} &= \|(-A(t-(t-s))^*)^{1+\theta}V(t;t-s,0)\|_{\calL(E_0^*)}
\\ & \leq C(t-s)^{-1-\theta}.
\end{align*}
Let $x\in D((-A(s))^{1+\theta})$ be arbitrary. Then
\begin{align*}
\|S(t,s)(-A(s))^{1+\theta}x\|_{E_0} &= \sup_{\|x^*\|_{E_0^*} \leq 1} |\lb S(t,s)(-A(s))^{1+\theta}x, x^*\rb | \\
 &= \sup_{\|x^*\|_{E_0^*} \leq 1} |\lb x, (-A(s)^*)^{1+\theta}S(t,s)^*x^* \rb | \\
 &\leq \|x\|_{E_0} \|(-A(s)^*)^{1+\theta}S(t,s)^*\|_{\calL(E_0^*)}
 \\ & \leq C(t-s)^{-1-\theta}\|x\|_{E_0}
\end{align*}
and \eqref{thrm-evolsys-frct-ineq4} follows. The proofs of \eqref{thrm-evolsys-frct-ineq42} and \eqref{thrm-evolsys-frct-ineq5} follow in the same way from \eqref{thrm-evolsys-frct-ineq1} and \eqref{thrm-evolsys-frct-ineq2c}, respectively.
\end{proof}

\section{Pathwise regularity of convolutions\label{sec:pathwisereg}}

In this section we will assume the following hypothesis.

\let\ALTERWERTA\theenumi
\let\ALTERWERTB\labelenumi
\def\theenumi{(H2)}
\def\labelenumi{(H2)}
\begin{enumerate}
  \item\label{hpths:ATplus} Both $(A(t,\omega))_{t\in [0,T],\omega\in \OO}$ and $(A(t,\omega)^*)_{t\in [0,T],\omega\in \OO}$ satisfy the (\AT)-conditions.
\end{enumerate}
\let\theenumi\ALTERWERTA
\let\labelenumi\ALTERWERTB

\subsection{A class of time independent spaces and interpolation}
The following hypothesis is needed to deal with the time-dependent domains in an efficient way.

\let\ALTERWERTA\theenumi
\let\ALTERWERTB\labelenumi
\def\theenumi{(H3)}
\def\labelenumi{(H3)}
\begin{enumerate}
  \item\label{hpths:interpol-spce-wk}
  There exist $\eta_+ \in (0,1]$ and $\eta_-\in (0,\mu^*+\nu^*-1)$ and two families of interpolation spaces $(\tilde{E}_\eta)_{\eta\in[0,\eta_+]}$ and $(\tilde{E}_\eta)_{\eta\in(-\eta_-,0]}$ such that
   \begin{enumerate}[(i)]
   \item For all $-\eta_- < \eta_4 \leq \eta_3 \leq 0\leq \eta_2\leq\eta_1\leq\eta_+$
   \[ \tilde{E}_{\eta_+} \hookrightarrow \tilde{E}_{\eta_1} \hookrightarrow \tilde{E}_{\eta_2} \hookrightarrow \tilde{E}_0 = E_0 \hookrightarrow \tilde{E}_{\eta_3} \hookrightarrow \tilde{E}_{\eta_4}.\]
   \item For all $\eta \in [0,\eta_+)$, $E_{\eta,1}^t \hookrightarrow \tilde{E}_\eta \hookrightarrow E_0$, with uniform constants in $(t,\omega)$.
   \item For $\eta \in (0, \eta_-)$, $E_0$ is dense in $\tilde{E}_{-\eta}$ and for all $x\in E_0$ and $\varepsilon>0$ one has
  \[ \|(-A(t))^{-\eta-\varepsilon} x\|_{E_0} \leq C \|x\|_{\tilde{E}_{-\eta}},\]
  where $C$ is independent of $(t,\omega)$.
   \end{enumerate}
\end{enumerate}
\let\theenumi\ALTERWERTA
\let\labelenumi\ALTERWERTB
If $E_1= D(A(t))$ is constant one can just take $\tilde{E}_{\eta} = (E_0, E_1)_{\eta,p}$ for some $p\in[1, \infty)$. Moreover, in particular it follows from (iii) that $(-A(t))^{-\eta-\varepsilon}$ has a unique continuous extension to a bounded operator from $E_0$ into $\tilde{E}_{-\eta}$. From Remark \ref{rmrk:evol-ineqs} it will become clear why we assume $\eta_-<\mu^*+\nu^*-1$.

\begin{remark}\label{rmrk:InterpolSpaceHypoth} \
\begin{enumerate}
\item If $A(t)$ is a differential operator with time dependent boundary conditions, then in general $E_\eta^t$ will be time dependent as well. In this case one typically takes $\tilde{E}_\eta$ to be the space obtained by real interpolation from $E_0$ and the space $E_1 \supset D(A(t))$, where $E_1$ is the space obtained by leaving out the boundary conditions.

\item Note that it is allowed to choose $\tilde{E}_{-\eta} = E_0$ for all $\eta\in (0,\eta_-)$. However, in the case $A(t)$ is a differential operator, one often takes $\tilde{E}_{\eta}$ to be an extrapolation space which allows a flexibility for the noise term for $\eta$ large enough.
\end{enumerate}
\end{remark}

\begin{remark}\label{rmrk:evol-ineqs}
Assume hypotheses \ref{hpths:ATplus} and \ref{hpths:interpol-spce-wk}. The following observation will be used throughout the rest of the paper. Let $\e > 0$, $a\in (0,\eta_+)$, $s<t$ and set $r = \tfrac{t+s}{2}$. By \eqref{array-cnt-emb1}, \eqref{array-cnt-emb2} and \eqref{thrm-evolsys-frct-ineq4}, for all $\theta \in [0,\eta_-)$ and $x\in \tilde{E}_{-\theta}$, we find that for $\e > 0$ small enough,
\begin{align*}
\|S(t,s)A(s) x\|_{\tilde{E}_a} &\leq C\|S(t,r)\|_{\calL(E_0, E_a^t)} \|S(r, s) (-A(s))^{1+\theta+\varepsilon} (-A(s))^{-\theta-\varepsilon} x\|_{E_0} \\
 &\leq C(t-s)^{-a-\e-1-\theta} \|x\|_{\tilde{E}_{-\theta}}.
\end{align*}
Note that here we use $\eta_-<\mu^*+\nu^*-1$. Similarly, with \eqref{thrm-evolsys-frct-ineq2b} we find that for all $\theta\in (0,\eta_-)$
\begin{equation}\label{eq:rmrk-ineqs-ineq2}
\|S(t,s)x\|_{\tilde{E}_a} \leq C(t-s)^{-a-\e-\theta} \|x\|_{\tilde{E}_{-\theta}},
\end{equation}
where in both estimates $C$ depends on $\omega$.
\end{remark}

The next lemma is taken from \cite[Lemma 2.3]{Veraar-SEE}, and this is the place where the assumption that $(\tilde{E}_\eta)_{\eta\in[0,\eta_+]}$ are interpolation spaces, is used.
\begin{lemma}\label{lemma:interpol}
Assume \ref{hpths:ATplus} and \ref{hpths:interpol-spce-wk}. Let $\alpha \in (0,\eta_+]$ and $\delta, \gamma >0 $ such that $\gamma + \delta \leq \alpha$. Then there exists a constant $C > 0$ depending on $\omega$, such that
\begin{align*}
\|S(t,r)x - S(s,r)x\|_{\tilde{E}_\delta} \leq C(t-s)^\gamma \|x\|_{E_{\alpha,1}^r}, \ \ \ 0\leq r\leq s\leq t\leq T, \ \ x\in E_{\alpha,1}^r.
\end{align*}
Moreover, if $x\in E_{\alpha,1}^r$, then $t\mapsto S(t,r)x \in C([r,T];\tilde{E}_\alpha)$.
\end{lemma}

\subsection{Sobolev spaces}
Let $X$ be a Banach space. For $\alpha\in (0,1)$, $p\in [1, \infty)$ and $a<b$, recall that a function $f:(a,b)\to X$ is said to be in the {\em Sobolev space} $W^{\alpha,p}(a,b;X)$  if $f\in L^p(a,b;X)$ and
\[[f]_{W^{\alpha,p}(a,b;X)} := \Big(\int_a^b \int_a^b \frac{\|f(t) - f(s)\|^p}{|t-s|^{\alpha p + 1}} \, ds \, dt \Big)^{1/p} <\infty.\]
Letting $\|f\|_{W^{\alpha,p}(a,b;X)} = \|f\|_{L^p(a,b;X)} + [f]_{W^{\alpha,p}(a,b;X)}$, this space becomes a Banach space. By symmetry one can write
\[\int_a^b \int_a^b B(t,s) \, ds \, dt = 2\int_a^b \int_a^t B(t,s)\, ds \, dt = 2\int_a^b \int_s^b B(t,s) \, dt \, ds\]
where $B(t,s) = \frac{\|f(t) - f(s)\|^p}{|t-s|^{\alpha p + 1}}$. This will be used often below.

A function $f:(a,b)\to X$ is said to be in the {\em H\"older space} $C^{\alpha}(a,b;X)$ if
\[[f]_{C^{\alpha}(a,b;X)} = \sup_{a<s<t<b} \frac{\|f(t) - f(s)\|}{|t-s|^\alpha}<\infty.\]
Letting $\|f\|_{C^{\alpha}(a,b;X)} = \sup\limits_{t\in (a,b)}\|f(t)\|_X + [f]_{C^{\alpha}(a,b;X)}$, this space becomes a Banach space. Moreover, every $f\in C^{\alpha}(a,b;X)$ has a unique extension to a continuous function $f:[a,b]\to X$.
For $p=\infty$, we also write $W^{\alpha,\infty}(0,T;X) = C^{\alpha}(0,T;X)$.

If $0<\alpha<\beta<1$, then trivially,
\begin{align*}
C^{\alpha}(a,b;X) \hookrightarrow W^{\alpha,p}(a,b;X), \ \  p\in [1, \infty].
\end{align*}
A well-known result is the following vector-valued Sobolev embedding:
\begin{equation}\label{eq:fractionalSobolevembReprForm}
W^{\alpha,p}(a,b;X) \hookrightarrow C^{\alpha-\frac{1}{p}}(a,b;X), \ \ \ \alpha>\frac1p.
\end{equation}
The proof in \cite[14.28 and 14.40]{Leoni09} and \cite[Theorem 8.2]{fracsob} for the scalar case extends to the vector-valued setting. Here the embedding means that each $f\in W^{\alpha,p}(a,b;X)$ has a version which is continuous
and this function lies in $C^{\alpha-\frac{1}{p}}(a,b;X)$.

\subsection{Regularity of generalized convolutions}

We can now present the first main result of this section. It gives a space-time regularity result for the abstract Cauchy problem:
\begin{equation}\label{eq:detproblem}
u'(t) = A(t) u(t) + f(t),  \ \ \ u(0) = 0.
\end{equation}
Recall that the solution is given by the convolution:
\[S*f(t) := \int_0^t S(t,\sigma) f(\sigma)\, d\sigma.\]

The next result extends \cite[Proposition 3.2]{Veraar-SEE}, where a space-time H\"older continuity result has been obtained.
\begin{theorem}\label{thm:F-Holder}
Assume \ref{hpths:ATplus} and \ref{hpths:interpol-spce-wk}. Let $\theta \in [0,\eta_-)$, $p \in [1,\infty)$ and $\delta, \lambda > 0$ such that $\delta + \lambda < \min\{1- \theta, \eta_+\}$. Suppose $f\in L^0(\OO;L^p(0,T;\tilde{E}_{-\theta}))$. Then the stochastic process
$S * f$ is in $L^0(\OO;W^{\lambda,p}(0,T;\tilde{E}_\delta))$ and satisfies
\[\|S * f\|_{W^{\lambda,p}(0,T;\tilde{E}_\delta)} \leq C\|f\|_{L^p(0,T;\tilde{E}_{-\theta})} \ \ \ a.s.,\]
where $C$ depends on $\omega$.
\end{theorem}
Maximal $L^p$-regularity results for \eqref{eq:detproblem} holds under additional assumptions on $(A(t))_{t\in [0,T]}$ and can be found in \cite{PortalStr}.

\begin{proof}
Let $\varepsilon>0$ be so small that $\delta+\lambda+\theta+2\varepsilon< 1$. By \eqref{eq:rmrk-ineqs-ineq2}, we find that
\begin{align*}
\|S &* f(t)\|_{\tilde{E}_{\delta}} \leq C\int_0^t (t-\sigma)^{-\delta-\theta-\varepsilon} \|f(\sigma)\|_{\tilde{E}_{-\theta}}  \,  d\sigma.
\end{align*}
Therefore, Young's inequality yields that
\begin{align*}
\|S * f\|_{L^p(0,T;\tilde{E}_\delta)}\leq C \|f\|_{L^p(0,T;\tilde{E}_{-\theta})}.
\end{align*}

Next we estimate the seminorm $[S*f]_{W^{\lambda,p}(0,T;E_{\delta})}$.
For $0\leq s < t \leq T$, we write
\begin{align*}
\|(S * f)(t) - (S * f)(s)\|_{\tilde{E}_{\delta}} & \leq \int_0^s \|(S(t,\sigma) - S(s,\sigma)) f(\sigma)\|_{\tilde{E}_{\delta}} \;d\sigma
\\ & \qquad + \int_s^t \|S(t,\sigma)f(\sigma)\|_{\tilde{E}_{\delta}} \;d\sigma.
\end{align*}

Observe that applying \eqref{thrm-evolsys-frct-ineq1} and \eqref{thrm-evolsys-frct-ineq42} yields
\begin{equation}\label{eq:Sfongelijkheid}
\begin{aligned}
\|&S(s,\sigma)f(\sigma)\|_{E_{\delta+\lambda+\e,1}^s} \leq \|S(s,\tfrac{\sigma+s}{2})\|_{\calL(E_0, E_{\delta+\lambda+\e,1}^s)} \|S(\tfrac{\sigma+s}{2},\sigma)f(\sigma)\|_{E_0} \\
&\leq (s-\sigma)^{-\delta-\lambda-\e} \|S(\tfrac{\sigma+s}{2},\sigma) ((-A(\sigma))^{\theta+\e})\|_{\calL(E_0)} \|(-A(\sigma))^{-\theta-\e}f(\sigma)\|_{E_0}  \\
&\leq (s-\sigma)^{-\delta - \lambda - \theta -2\e} \|f(\sigma)\|_{\tilde{E}_{-\theta}}.
\end{aligned}
\end{equation}
Hence by Lemma \ref{lemma:interpol} we obtain
\begin{align*}
\int_0^s \|(S(t,\sigma) &- S(s,\sigma))f(\sigma)\|_{\tilde{E}_{\delta}}\;d\sigma \leq C(t-s)^{\lambda+\e} \int_0^s \|S(s,\sigma)f(\sigma)\|_{E_{\delta+\lambda+\e,1}^s} \, d\sigma
\\&\leq C(t-s)^{\lambda+\e} \int_0^s (s-\sigma)^{-\delta -\lambda- \theta -2\e} \|f(\sigma)\|_{\tilde{E}_{-\theta}} \;d\sigma.
\end{align*}
Now it follows from integration over $t$ and then Young's inequality that
\begin{align*}
\int_0^T &\int_s^T (t-s)^{-1-\lambda p} \Big(\int_0^s \|(S(t,\sigma) - S(s,\sigma))f(\sigma)\|_{\tilde{E}_{\delta}}\;d\sigma\Big)^p \, dt \, ds
\\ & \leq C \int_0^T \int_s^T (t-s)^{-1 + \varepsilon p } \Big(\int_0^s (s-\sigma)^{-\delta - \lambda - \theta -2\e}\|f(\sigma)\|_{\tilde{E}_{-\theta}} \;d\sigma\Big)^{p} \, dt \, ds
\\ & \leq C  \int_0^T  \Big(\int_0^s (s-\sigma)^{-\delta - \lambda - \theta -2\e}\|f(\sigma)\|_{\tilde{E}_{-\theta}} \;d\sigma\Big)^{p} \, ds
\\ & \leq C  \|f\|_{L^p(0,T;\tilde{E}_{-\theta})}^p.
\end{align*}

For the other term by \eqref{thrm-evolsys-frct-ineq1}  we obtain
\begin{align*}
\int_s^t &\|S(t,\sigma)f(\sigma)\|_{\tilde{E}_{\delta}}\;d\sigma \leq \int_0^t \one_{(s,t)}(\sigma) (t-\sigma)^{-\delta - \theta-\e}\|f(\sigma)\|_{\tilde{E}_{-\theta}}\;d\sigma.
\end{align*}
Integrating over $s\in (0,t)$, it follows from Minkowski's inequality that
\begin{align*}
\Big(\int_0^t &\Big( (t-s)^{-\frac{1}{p} - \lambda}\int_0^t \one_{(s,t)}(\sigma) \|S(t,\sigma)f(\sigma)\|_{\tilde{E}_{\delta}}\;d\sigma \Big)^{p} \, ds\Big)^{1/p}
\\ & \leq C  \Big(\int_0^t \Big(\int_0^t (t-s)^{-\frac{1}{p} - \lambda} \one_{(s,t)}(\sigma) (t-\sigma)^{-(\delta + \theta+ \e)}\|f(\sigma)\|_{\tilde{E}_{-\theta}}\;d\sigma \Big)^{p} \, ds \Big)^{1/p}
\\ & \leq C \int_0^t \Big(\int_0^t (t-s)^{-1 - \lambda p} \one_{(s,t)}(\sigma) (t-\sigma)^{-(\delta + \theta+\e)p}\|f(\sigma)\|_{\tilde{E}_{-\theta}}^p\;ds \Big)^{1/p} \, d\sigma
\\ & \leq C \int_0^t (t-\sigma)^{-(\delta + \theta+\e+\lambda)}\|f(\sigma)\|_{\tilde{E}_{-\theta}}  \, d\sigma.
\end{align*}
Taking $p$-th moments in $t\in (0,T)$, it follows from Young's inequality that
\begin{align*}
\int_0^T \int_0^t &\Big( (t-s)^{-\frac{1}{p} - \lambda}\int_0^t \one_{(s,t)}(\sigma) \|S(t,\sigma)f(\sigma)\|_{\tilde{E}_{\delta}}\;d\sigma \Big)^{p} \, ds\, dt
\\ & \leq C \int_0^T \Big(\int_0^t (t-\sigma)^{-(\delta + \theta+\e+\lambda)}\|f(\sigma)\|_{\tilde{E}_{-\theta}}  \, d\sigma\Big)^{p} \, dt
\leq C  \|f\|_{L^p(0,T;\tilde{E}_{-\theta})}^p.
\end{align*}
Combining the estimates, the result follows.
\end{proof}

The second main result of this section gives a way to obtain pathwise regularity of the solution to the stochastic Cauchy problem
\[du = A(t) u(t) \, dt + G \, d W.\]
given by \eqref{eq:reprformintro}. For details on this we refer to Section \ref{sec:repr} below.

Recall the convention that for a Banach space $X$, we put $W^{\alpha,\infty}(0,T;X) = C^{\alpha}(0,T;X)$.
\begin{theorem}\label{thm:SP-Holder-ctu}
Assume \ref{hpths:ATplus} and  \ref{hpths:interpol-spce-wk}. Let $p\in (1, \infty]$, $\theta \in [0,\eta_-)$, $\alpha > \theta$ and let $\delta,\lambda >0 $ such that $\delta+\lambda < \min\{\alpha - \theta, \eta_+\}$. Let $f \in L^0(\OO;W^{\alpha,p}(0,T;\tilde{E}_{-\theta}))$.
The following assertions hold:
\begin{enumerate}
  \item The stochastic process $\zeta$ defined by
\begin{align}
  \zeta(t) := \int_0^t S(t,\sigma)A(\sigma)(f(t)-f(\sigma))\;d\sigma
\end{align}
belongs to $L^0(\OO;W^{\lambda,p}(0,T;\tilde{E}_{\delta}))$ and
\[\|\zeta\|_{W^{\lambda,p} (0,T;\tilde{E}_{\delta})} \leq C \|f\|_{W^{\alpha,p}(0,T;\tilde{E}_{-\theta})} \ \ \ a.s.,\]
where $C$ depends on $\omega$.
  \item If $\alpha>1/p$ and the continuous version of $f$ satisfies $f(0) = 0$, then $\tilde{\zeta} := S(t,0)f(t)$ belongs to $L^0(\OO;W^{\lambda,p}(0,T;\tilde{E}_{\delta}))$ and
  \[\|\tilde{\zeta}\|_{W^{\lambda,p} (0,T;\tilde{E}_{\delta})} \leq C \|f\|_{W^{\alpha,p}(0,T;\tilde{E}_{-\theta})} \ \ \ a.s.,\]
where $C$ depends on $\omega$.
\end{enumerate}
\end{theorem}

Note that in (2) the continuous version of $f$ exists by \eqref{eq:fractionalSobolevembReprForm}.

\begin{proof}
The proofs in the case $p=\infty$ are much simpler and we focus on the case $p\in (1, \infty)$. Let $\beta = \delta +\theta +\varepsilon$.

(1). Let $\e > 0$ be so small that $\beta + \lambda +\e< \alpha $.
Write $\Delta_{ts}f = f(t)-f(s)$. First we estimate the $L^p(0,T;\tilde{E}_{\delta})$-norm  of $\zeta$. Note that by Remark \ref{rmrk:evol-ineqs},
\[\|S(t,\sigma)A(\sigma)\Delta_{t\sigma}f \|_{\tilde{E}_{\delta}}\leq C(t-\sigma)^{-1-\beta}\|\Delta_{t\sigma}f\|_{\tilde{E}_{-\theta}}.\]
Therefore, by Holder's inequality applied with measure $(t-\sigma)^{-1+\varepsilon}\,d\sigma$ we find that
\begin{align*}
\|\zeta\|_{L^p(0,T;\tilde{E}_{\delta})}^p &
\leq C\int_0^T \Big( \int_0^t \frac{\|\Delta_{t\sigma}f\|_{\tilde{E}_{-\theta}}}{(t-\sigma)^{1+\beta}} \,d\sigma\Big)^{p} \, dt
\\ & \leq C\int_0^T \int_0^t \frac{\|\Delta_{t\sigma}f\|_{\tilde{E}_{-\theta}}^p}{(t-\sigma)^{1+(\beta+\varepsilon - \frac{\varepsilon}{p})p}} \,d\sigma \, dt
\\ & \leq C \|f\|_{W^{\alpha,p}(0,T;\tilde{E}_{-\theta})}.
\end{align*}

Observe that
\begin{align*}
\|\zeta(t) - \zeta(s)\|_{\tilde{E}_\delta} & \leq \int_0^s \| (S(t,\sigma)A(\sigma)\Delta_{t \sigma}f-S(s,\sigma)A(\sigma)\Delta_{s \sigma}f)\|_{\tilde{E}_{\delta}}\;d\sigma \\ & \qquad + \int_s^t \|S(t,\sigma)A(\sigma)\Delta_{t \sigma}f\|_{\tilde{E}_{\delta}}\;d\sigma = T_1(s,t) + T_2(s,t).
\end{align*}

We estimate the $[\cdot]_{W^{\lambda,p}}$-seminorm of each of the terms separately. For $T_2$ note that by Remark \ref{rmrk:evol-ineqs},
\begin{align*}
\|S(t,\sigma)A(\sigma)\Delta_{t \sigma}f\|_{\tilde{E}_\delta} &\leq C(t-\sigma)^{-\beta - 1}\|\Delta_{t \sigma}f\|_{\tilde{E}_{-\theta}}=:g(\sigma,t).
\end{align*}
Therefore, it follows from the Hardy--Young inequality (see \cite[p. 245-246]{Hardy}) that
\begin{align*}
\int_0^t (t-s)^{-\lambda p  -1}  \Big( & \int_s^t \|S(t,\sigma)A(\sigma)\Delta_{t \sigma}f\|_{\tilde{E}_{\delta}}\;d\sigma \Big)^p\,ds
\\ & \leq \int_0^t (t-s)^{-\lambda p  -1} \Big( \int_s^t  g(\sigma,t) \;d\sigma\Big)^{p} \, ds
\\ & = \int_0^t (t-s)^{-\lambda p  -1} \Big( \int_0^{t-s}  g(t-\tau,t) \;d\tau\Big)^{p} \, ds
\\ & = \int_0^t r^{-\lambda p -1}  \Big( \int_0^{r}  g(t-\tau,t) \;d\tau\Big)^{p} \, dr
\\ & \leq C\int_0^t r^{p-\lambda p  -1} g(t-r,t)^p \, dr
\\ & = C\int_0^t \frac{g(s,t)^p}{(t-s)^{-p+\lambda p +1} } \, ds.
\end{align*}
Integrating with respect to $t\in (0,T)$ and using the definition of $g$ we find that
\begin{align*}
\int_0^T \int_0^t \frac{T_{2}(s,t)^p}{(t-s)^{\lambda p +1}} \, ds \, dt  & \leq C \int_0^T \int_0^t  \frac{(t-s)^{-(\beta+ 1)p} \|\Delta_{t s}f\|_{\tilde{E}_{-\theta}}^p}{(t-s)^{-p+\lambda p +1} } \, ds \, dt
\\ & \leq C \int_0^T \int_0^t  \frac{\|\Delta_{t s}f\|_{\tilde{E}_{-\theta}}^p}{(t-s)^{\alpha p +1} } \leq C\|f\|_{W^{\alpha,p}(0,T;\tilde{E}_{-\theta})}^p.
\end{align*}

For $T_1$, we can write
\begin{align}\label{eqn:regul-intsplit3}
\begin{split}
T_1(s,t) & \leq \int_0^s\|S(t,\sigma)A(\sigma)\Delta_{ts}f\|_{\tilde{E}_{\delta}}\;d\sigma
\\ & \qquad + \int_0^s \|(S(t,\sigma)-S(s,\sigma))A(\sigma)\Delta_{s \sigma}f\|_{\tilde{E}_{\delta}}\;d\sigma
\\ & =  T_{1a}(s,t) + T_{1b}(s,t).
\end{split}
\end{align}
For $T_{1a}$, by Remark \ref{rmrk:evol-ineqs} we have
\begin{align*}
T_{1a}(s,t)\leq C \int_0^s (t-\sigma)^{-1-\beta} \|\Delta_{t s}f\|_{\tilde{E}_{-\theta}} \;d\sigma
\leq C(t-s)^{-\beta} \|\Delta_{t s}f\|_{\tilde{E}_{-\theta}}.
\end{align*}
It follows that
\begin{align*}
\int_0^T \int_0^t \frac{T_{1a}(s,t)^p}{(t-s)^{\lambda p +1}} \, ds \, dt & \leq C \int_0^T \int_0^t
\frac{\|\Delta_{t s}f\|_{\tilde{E}_{-\theta}}^p}{(t-s)^{(\beta+ \lambda)p +1} }\, ds \, dt
\\ & \leq C \|f\|_{W^{\alpha,p}(0,T;\tilde{E}_{-\theta})}^p.
\end{align*}

To estimate $T_{1b}$ let $\gamma = \alpha-\varepsilon-\beta$. By Lemma \ref{lemma:interpol}, Lemma \ref{lem:-evolsys-frct-ineq} and Lemma \ref{lem:diff-prop-evolsys},
\begin{align*}
\|(S&(t,\sigma) -S(s,\sigma))A(\sigma)\Delta_{s \sigma}f\|_{\tilde{E}_{\delta}} \leq \|(S(t,s)-I)S(s,\sigma)A(\sigma) \Delta_{s \sigma}f\|_{\tilde{E}_{\delta}}
\\ & \leq C(t-s)^{\gamma} \|S(s,\sigma)(-A(\sigma))^{1+\theta+\varepsilon}(-A(\sigma))^{-\theta-\e}\Delta_{s \sigma}f\|_{E_{\delta + \gamma,1}^s}
\\ & \leq  C(t-s)^{\gamma} \|S(s,\tau)\|_{\calL(E_0, E_{\delta + \gamma,1}^s)} \|S(\tau, \sigma)A(\sigma)^{1+\theta+\varepsilon}\|_{\calL(E_0)} \|\Delta_{s \sigma}f\|_{\tilde{E}_{-\theta}}
\\ & \leq C(t-s)^{\gamma} (s-\sigma)^{-1-\gamma-\beta} \|\Delta_{s \sigma}f\|_{\tilde{E}_{-\theta}},
\end{align*}
with $\tau = (s+\sigma)/2$. It follows that from H\"older's inequality that
\begin{align*}
T_{1b}(s,t)^p &\leq C(t-s)^{\gamma p} \Big(\int_0^s (s-\sigma)^{-1-\gamma-\beta}  \|\Delta_{s \sigma}f\|_{\tilde{E}_{-\theta}} \, d\sigma\Big)^p
\\ & \leq C (t-s)^{\gamma p} h(s)^p \int_0^s \frac{\|\Delta_{s \sigma}f\|_{\tilde{E}_{-\theta}}^p}{(s-\sigma)^{\alpha p + 1}} \, d\sigma,
\end{align*}
where by the choice of $\gamma$, the function $h(s)$ satisfies
\begin{align*}
h(s) & = \Big(\int_0^s \Big[(s-\sigma)^{-1-\gamma-\beta+\alpha +\frac1p} \Big]^{p'} \, d\sigma\Big)^{1/p'}\leq C.
\end{align*}
Using Fubini's theorem and $\gamma>\lambda$ we can conclude that
\begin{align*}
\int_0^T \int_s^T \frac{T_{1b}(s,t)^p}{(t-s)^{\lambda p +1}} \, dt \, ds
& \leq C\int_0^T \int_0^s \int_s^T (t-s)^{-\lambda p - 1+\gamma p} \, dt \frac{\|\Delta_{s \sigma}f\|_{\tilde{E}_{-\theta}}^p}{(s-\sigma)^{\alpha p + 1}} \, d\sigma \, ds
\\
& \leq C \|f\|_{W^{\alpha,p}(0,T;\tilde{E}_{-\theta})}^p.
\end{align*}
and this finishes the proof of (1).

To prove (2), let $\e > 0$ be such that $\frac1p<\beta + \lambda +2\e< \alpha$. To estimate $[\tilde{\zeta}]_{W^{\lambda,p}(0,T;\tilde{E}_{\delta})}$ we write
\[ \|\tilde{\zeta}(t) - \tilde{\zeta}(s)\|_{\tilde{E}_\delta} \leq \|S(t,0)\Delta_{t s}f\|_{\tilde{E}_\delta} + \|S(t,0)-S(s,0)f(s)\|_{\tilde{E}_\delta}.\]
By Remark \ref{rmrk:evol-ineqs},
\begin{align*}
\|S(t,0)\Delta_{t s}f\|_{\tilde{E}_\delta} &\leq Ct^{-\beta} \|\Delta_{t s}f\|_{\tilde{E}_{-\theta}}
\leq C(t-s)^{-\beta} \|\Delta_{t s}f\|_{\tilde{E}_{-\theta}}.
\end{align*}
It follows that
\begin{align*}
\int_0^T \int_0^t  \frac{\|S(t,0)\Delta_{t s}f\|_{\tilde{E}_\delta}^p}{(t-s)^{\lambda p + 1}} \, ds \, dt \leq C\|f\|_{W^{\beta+\lambda}(0,T;\tilde{E}_{-\theta})} \leq C\|f\|_{W^{\alpha}(0,T;\tilde{E}_{-\theta})}.
\end{align*}

By \cite[Theorem 5.4]{fracsob} we can find an extension $F\in W^{\alpha,p}(\R;\tilde{E}_{-\theta})$ of $f$ such that
\[\|F\|_{W^{\alpha,p}(\R;\tilde{E}_{-\theta})} \leq C \|f\|_{W^{\alpha,p}(0,T;\tilde{E}_{-\theta})}.\]
Moreover, multiplying $F$ by a suitable smooth cut-off function we can assume that additionally $F = 0$ on $[T+1,\infty)$.

We have by Lemma \ref{lemma:interpol} and \eqref{eq:Sfongelijkheid},
\begin{align*}
\|(S(t,0)-S(s,0))f(s)\|_{\tilde{E}_\delta} &= \|(S(t,s) - S(s,s))S(s,0)f(s)\|_{\tilde{E}_\delta}
\\ & \leq C(t-s)^{\lambda+\varepsilon}\|S(s,0)f(s)\|_{E^s_{\delta +\lambda+\varepsilon,1}} \\ & \leq C(t-s)^{\lambda+\varepsilon} \frac{\|f(s)\|_{\tilde{E}_{-\theta}}}{s^{\beta+\lambda +2\varepsilon}}.
\end{align*}
Therefore, we find
\begin{align*}
\int_0^T \int_s^T  &\frac{\|S(t,0)-S(s,0)f(s)\|_{\tilde{E}_\delta}^p}{(t-s)^{\lambda p + 1}} \, dt \, ds  \leq C\int_0^T \int_s^T (t-s)^{-1+\varepsilon p} \, dt    \frac{\|f(s)\|_{\tilde{E}_{-\theta}}^p}{s^{(\beta+\lambda +2\varepsilon)p}} \, ds
\\ & \leq C\int_0^T \frac{\|f(s)\|_{\tilde{E}_{-\theta}}^p}{s^{(\beta+\lambda +2\varepsilon)p}}  \, ds
\leq C\int_{\R_+} \frac{\|F(s)\|_{\tilde{E}_{-\theta}}^p}{s^{(\beta+\lambda +2\varepsilon)p}}  \, ds.
\end{align*}
Applying the fractional Hardy inequality (see \cite[Theorem 2]{KMP}; here we use $\frac1p<\beta + \lambda +2\e$) and elementary estimates we find that the latter is less or equal than
\begin{align*}
C\|F\|_{W^{\beta+\lambda+2\varepsilon,p}(\R_+;\tilde{E}_{-\theta})}^p
& \leq C\|F\|_{W^{\beta+\lambda+2\varepsilon,p}(\R;\tilde{E}_{-\theta})}^p
\\ & \leq C\|f\|_{W^{\beta+\lambda+2\varepsilon,p}(0,T;\tilde{E}_{-\theta})}^p
\leq C\|f\|_{W^{\alpha,p}(0,T;\tilde{E}_{-\theta})}^p.
\end{align*}

Finally we estimate the $L^p$-norm of $\tilde{\zeta}$. From Remark \ref{rmrk:evol-ineqs} we see that
\begin{align*}
\|\tilde{\zeta}(t)\|_{\tilde{E}_{\delta}}\leq C t^{-\beta} \|f(t)\|_{\tilde{E}_{-\theta}}.
\end{align*}
Taking $L^p$-norms and applying the fractional Hardy inequality as before we obtain
\begin{align*}
\|\tilde{\zeta}\|_{L^p(0,T;\tilde{E}_{\delta})}^p\leq C \int_0^T t^{-\beta p} \|f(t)\|_{\tilde{E}_{-\theta}}^p \, dt \leq C\|f\|_{W^{\alpha,p}(0,T;\tilde{E}_{-\theta})}^p.
\end{align*}
This completes the proof of (2).
\end{proof}

\section{Pathwise mild solutions\label{sec:repr}}

In this section we will introduce a new solution formula for equations of the form:
\begin{equation}\label{eq:SEEsimple}
\left\{
    \begin{array}{ll}
        dU(t) &= A(t)U(t)\;dt + G\;dW(t), \\
        U(0) &= 0.
    \end{array}
\right.
\end{equation}
Here $W$ is an $H$-cylindrical Brownian motion and $G:[0,T]\times\OO\to \calL(H,E_0)$ is adapted and strongly measurable. The initial condition is taken zero for notational convenience. Furthermore, $(A(t))_{t\in[0,T]}$ satisfies the (\AT)-conditions as introduced before. At first sight one would expect that the solution to \eqref{eq:SEEsimple} is given by
\begin{equation}\label{eq:mild}
U(t) = \int_0^t S(t,s) G(s)\, d W(s).
\end{equation}
However, in general $s\mapsto S(t,s)$ is only $\F_t$-measurable (see Proposition \ref{prop:evolsys-adapted} and Example \ref{ex:noadapted}). Therefore, the stochastic integral does not exist in the It\^o sense. We will give another representation formula which provides an alternative to mild solutions to \eqref{eq:SEEsimple}. We say $U$ is a {\em pathwise mild solution} to \eqref{eq:SEEsimple} if almost surely for almost all $t\in [0,T]$,
\begin{align}\label{eq:reprform}
U(t) = - \int_0^t S(t,s) A(s) I(\one_{(s,t)}G) \, ds + S(t,0) I(\one_{(0,t)} G),
\end{align}
where $I(\one_{(s,t)}G) = \int_s^t G \, d W$. If  $p\in (2, \infty)$ is large enough, then \eqref{eq:reprform} is defined in a pointwise sense and pathwise continuous (see Theorem \ref{thm:SP-Holder-ctu}).

A better name for the solution \eqref{eq:reprform} would be a `pathwise singular representation of the mild solution'. Pathwise appears in the name, since the representation formula \eqref{eq:reprform} is defined in a pathwise sense, but we prefer to use the abbreviated version.

Clearly, there is no adaptedness issue in \eqref{eq:reprform} as the evolution family is only used in integration with respect to the Lebesgue measure. Moreover, the solution $U$ will be adapted and measurable. A difficulty in the representation formula \eqref{eq:reprform} is that usually the kernel $S(t,s) A(s)$ has a singularity of order $(t-s)^{-1}$, but we will see below that $I(\one_{(s,t)}G)$ is small enough for $s$ close to $t$ to make the integral in \eqref{eq:reprform} convergent. Moreover, we will show that the usual parabolic regularity results hold.

In Section \ref{subsec:stochint} we first repeat some basic results from stochastic integration theory. In Section \ref{subsec:motStoch} we show that in the bounded case the pathwise mild solution \eqref{eq:reprform} yields the right solution. The space-time regularity of $U$ defined by \eqref{eq:reprform} is studied in Section \ref{subsec:regStoch}. In Section \ref{subsec:weaksol} we show that a pathwise mild solution is a variational and weak solution and vice versa. Finally in Section \ref{subsec:mild} we prove the equivalence with a forward mild solution \eqref{eq:mild} (based on forward integration). The latter will not be used in the rest of the paper, but provides an interesting connection.

In this section we assume the hypotheses \ref{asmp:A-mble-adptd}-\ref{hpths:interpol-spce-wk}
and we impose a further condition on the spaces in \ref{hpths:interpol-spce-wk}.
\let\ALTERWERTA\theenumi
\let\ALTERWERTB\labelenumi
\def\theenumi{(H4)}
\def\labelenumi{(H4)}
\begin{enumerate}
  \item\label{hpths:interpol-spce-wkextra}
  The spaces $(\tilde{E}_\eta)_{\eta\in(-\eta_-,\eta_+]}$ from \ref{hpths:interpol-spce-wk} all have {\umd} and type $2$.
\end{enumerate}
\let\theenumi\ALTERWERTA
\let\labelenumi\ALTERWERTB
Details on {\umd} and type can be found in \cite{Bu3} and \cite{DJT}, respectively.

\subsection{Stochastic integration\label{subsec:stochint}}
Below we briefly repeat a part of the stochastic integration theory in {\umd} spaces $E$ with type $2$. For details on stochastic integration in such spaces or more generally martingale type $2$ spaces, we refer to \cite{Brzezniak, NVWsurvey, Ondrejat04, Seidler} and the references therein.

For spaces $E$ with {\umd} one can develop an analogue of It\^o's theory of the stochastic integral (see \cite{NVW1}). To be more precise: one can precisely characterize which adapted and strongly measurable $G:[0,T]\times\OO\to \calL(H,E)$ are stochastically integrable. Moreover, two-sided estimates can be obtained. If additionally the space $E$ has type $2$, then there exists an easy subspace of stochastically integrable processes.
Indeed, for every adapted and strongly measurable $G \in L^0(\OO;L^2(0,T;\g(H,E)))$, the stochastic integral process $\Big(\int_0^t G \, d W\Big)_{t\in [0,T]}$ exists and is pathwise continuous. For convenience we write
\[I(G) = \int_0^T G \, d W \ \ \text{and} \ \ J(G)(t) = I(\one_{[0,t]}G) \ \ t\in [0,T].\]
Moreover, for all $p\in(0,\infty)$, there exists a constant $C$ independent of $G$ such that the following one-sided estimate holds:
\begin{align}\label{eq:JGest}
\|J(G)\|_{L^p(\OO;C([0,T];E))} \leq C \|G\|_{L^p(\OO;L^2(0,T;\g(H,E)))}.
\end{align}
Here, $\g(H,E)$ is the space of $\gamma$-radonifying operators $R:H\to E$. For details on $\gamma$-radonifying operators we refer to \cite{Neerven-Radon,NVWsurvey}.

One can deduce Sobolev regularity of the integral process (see \cite{PV-forward}).
\begin{proposition}\label{prop:type2prop}
Assume $E$ has type $2$, let $p \in [2, \infty)$ and $0<\alpha<\frac12$.
If $G$ is an adapted process in $L^0(\OO;L^p(0,T;\g(H,E)))$, then $J(G)\in L^0(\OO;W^{\alpha,p}(0,T; E))$. Furthermore,
there exists a constant $C_T$ such that independent of $G$ such that
\[\|J(G)\|_{L^p(\OO;W^{\alpha,p}(0,T; E))} \leq C_T \|G\|_{L^p(\OO;L^p(0,T;\g(H,E)))},\]
where $C_T\to 0$ as $T\downarrow 0$. Moreover, if $(G_n)_{n=1}^\infty$ is a sequence of adapted processes in $L^0(\OO;L^p(0,T;\g(H,E)))$, then
\begin{align*}
G_n\to G \ &\text{in $L^0(\OO;L^p(0,T;\g(H,E)))$} \\&\qquad\qquad \Longrightarrow \
J(G_n)\to J(G)\ \text{in $L^0(\OO;W^{\alpha,p}(0,T; E))$}.
\end{align*}
\end{proposition}
By \eqref{eq:fractionalSobolevembReprForm} one can also derive H\"older regularity and convergence in the H\"older norm in the case $\alpha>1/p$.

\subsection{Motivation in the bounded case\label{subsec:motStoch}}
Below we will show that in the special case of bounded $A$, the pathwise mild solution $U$ defined by \eqref{eq:reprform}, is a solution to \eqref{eq:SEEsimple}.
\begin{proposition}\label{prop:equivstrongrepr}
Assume $A$ is an adapted process in $L^0(\OO;C([0,T];\calL(E_0))$. If $G\in L^0(\OO;L^2(0,T;\g(H,E_0)))$ is adapted, then $U$ defined by \eqref{eq:reprform} is adapted and satisfies
\[U(t) = \int_0^t A(s) U(s) \, ds + \int_0^t G(s)\, d W(s).\]
\end{proposition}
The above result is only included to show that \eqref{eq:reprform} leads to the ``right''solution. In the case $A$ is bounded, one can construct the solution as in the classical finite dimensional situation.

\begin{proof}
By \cite[Theorem 5.2]{Pazy}, $(A(t,\omega))_{t\in [0,T]}$ generates a unique continuous evolution family $(S(t,s,\omega))_{0\leq s\leq t\leq T}$ and pointwise in $\OO$ the following identities hold
\[\frac{\partial}{\partial t} S(t,s) = A(t) S(t,s) \ \ \text{and} \ \ \frac{\partial}{\partial s} S(t,s) = -S(t,s)A(s), \ \ 0\leq s\leq t\leq T.\]
Moreover, from the construction in \cite[Theorem 5.2]{Pazy} one readily checks that for each $0\leq s\leq t\leq T$, $\omega\mapsto S(t,s,\omega)$ is $\F_t$-measurable and thus $U$ defined by \eqref{eq:reprform} is adapted.
It follows that
\begin{align*}
U(t) &= -\int_0^t S(t,s) A(s) I(\one_{(0,t)}G) \, ds \\&\qquad\qquad + \int_0^t S(t,s) A(s) I(\one_{(0,s)}G) \, ds + S(t,0) I(\one_{(0,t)} G)
\\ & = I(\one_{(0,t)}G)  + \int_0^t S(t,s) A(s) I(\one_{(0,s)}G) \, ds.
\end{align*}
Therefore, by Fubini's theorem we obtain
\begin{align*}
\int_0^t A(r) U(r) \, dr & =  \int_0^t A(r) I(\one_{(0,r)}G) \, dr  + \int_0^t \int_0^r A(r) S(r,s) A(s) I(\one_{(0,s)}G) \, ds\, dr
\\ & = \int_0^t A(r) I(\one_{(0,r)}G) \, dr  + \int_0^t \int_s^t A(r) S(r,s) A(s) I(\one_{(0,s)}G) \, dr \, ds
\\ & = \int_0^t S(t,s) A(s) I(\one_{(0,s)}G) \, ds.
\end{align*}
Combining both identities, the result follows.
\end{proof}

\begin{remark}
In the general case that $A$ is unbounded, the integrals in the above proof might diverge and one needs to argue in a different way. However, if $G$ takes values in a suitable subspace of the domains of $A$, and under integrability assumptions in $s\in (0,T)$, one can repeat the above calculation in several situations.
\end{remark}

\subsection{Regularity\label{subsec:regStoch}}

As a consequence of the previous results we will now derive a pathwise regularity result for the pathwise mild solution $U$ given by \eqref{eq:reprform}.

\begin{theorem}\label{thm:pathregularity}
Let $p\in (2, \infty)$ and let $\theta \in [0,\min\{\eta_-,\tfrac12\})$. Let $\delta, \lambda > 0$ be such that $\delta + \lambda < \min\{\frac12 - \theta, \eta_+\}$. Suppose $G\in L^0(\OO;L^p(0,T;\g(H,\tilde{E}_{-\theta})))$ is adapted.
The process $U$ given by \eqref{eq:reprform} is adapted and is in $L^0(\OO;W^{\lambda,p} (0,T;\tilde{E}_{\delta}))$. Moreover, for every $\alpha\in (\lambda+\delta+\theta, \tfrac12)$, there is a mapping $C:\OO\to \R_+$ which only depends on $\delta, \lambda, p$ and the constants in \ref{asmp:A-mble-adptd}-\ref{hpths:interpol-spce-wkextra} such that
\[\|U\|_{W^{\lambda,p} (0,T;\tilde{E}_{\delta})} \leq C \|J(G)\|_{W^{\alpha,p} (0,T;\tilde{E}_{-\theta})}.\]
\end{theorem}
Note that $J(G)\in W^{\alpha,p} (0,T;\tilde{E}_{-\theta})$ a.s.\ by Proposition \ref{prop:type2prop}.

\begin{proof}
Let $\alpha\in (\lambda+\delta+\theta, \tfrac12)$. By Proposition \ref{prop:type2prop}, $J(G)$ belongs to the space $L^0(\OO;W^{\alpha,p} (0,T;\tilde{E}_{-\theta}))$. Therefore, by Theorem \ref{thm:SP-Holder-ctu} we find the required regularity and estimate for the paths of $U$. The measurability and adaptedness of $U$ follows from Proposition \ref{prop:evolsys-adapted} and approximation.
\end{proof}

\subsection{Weak solutions\label{subsec:weaksol}}
In this section we assume \ref{asmp:A-mble-adptd}-\ref{hpths:interpol-spce-wkextra}.
Formally, applying a functional $x^*\in E_0^*$ on both sides of \eqref{eq:SEEsimple} and integrating gives
\begin{align}\label{eq:weak1}
\lb U(t), x^*\rb &= \int_0^t \lb U(s), A(s)^* x^*\rb \;ds + \int_0^t G(s)^*x^* \;dW(s),
\end{align}
where the last expression only makes sense if $x^*\in D(A(s)^*)$ for almost all $s\in (0,T)$ and $\omega\in \OO$, and $s\mapsto \lb U(s),A(s)^* x^*\rb$ is in $L^1(0,T)$ almost surely.

In the case the domains $D(A(t,\omega))$ depend on $(t,\omega)$, it is more natural to use $(t,\omega)$-dependent functionals $\varphi:[0,t]\times\OO\to E_0^*$ to derive a weak formulation of the solution. Here $\varphi$ will be smooth in space and time, but will not be assumed to be adapted. Formally, applying the product rule to differentiate and then integrate the differentiable function $\lb U(t) - I(\one_{(0,t)}G), \varphi(t)\rb$, one derives that
\begin{align*}
\lb &U(t) - I(\one_{(0,t)}G), \varphi(t)\rb \\& = \int_0^t \lb A(s)U(s), \varphi(s)\rb \;ds + \int_0^t \lb U(s) - I(\one_{(0,s)}G),\varphi'(s)\rb \;ds
\\ & = \int_0^t \lb U(s), A(s)^*\varphi(s)\rb \;ds + \int_0^t \lb U(s),\varphi'(s)\rb \;ds - \int_0^t \lb I(\one_{(0,s)}G),\varphi'(s)\rb \;ds.
\end{align*}
Adding the stochastic integral term to both sides yields
\begin{equation}\label{eq:var1}
\begin{aligned}
\lb U(t), \varphi(t)\rb & = \int_0^t \lb U(s), A(s)^*\varphi(s)\rb \;ds + \int_0^t \lb U(s),\varphi'(s)\rb \;ds \\ & \qquad - \int_0^t \lb I(\one_{(0,s)}G),\varphi'(s)\rb \;ds + \lb I(\one_{(0,t)}G), \varphi(t)\rb.
\end{aligned}
\end{equation}
Clearly, \eqref{eq:var1} reduces to \eqref{eq:weak1} if $\varphi \equiv x^*$. Below we will show that the pathwise mild solution \eqref{eq:reprform} is equivalent to \eqref{eq:var1}. Moreover, in the case the domains are constant in time, both are equivalent to \eqref{eq:weak1}. Therefore, this provides the appropriate weak setting to extend the equivalence of Proposition \ref{prop:equivstrongrepr}.

First we define a suitable space of test functions.
\begin{definition}
For $t\in [0,T]$ and $\beta\geq 0$ let $\Gamma_{t,\beta}$ be the subspace of all $\varphi \in L^0(\OO;C^{1}([0,t];E_0^*))$ such that
\begin{enumerate}[(1)]
\item for all $s\in [0,t)$ and $\omega\in \OO$, we have $\varphi(s) \in D(((-A(s))^{\beta+1})^*)$ and $\varphi'(s)\in D(((-A(s))^{\beta})^*)$.
\item the process $s\mapsto A(s)^*\varphi(s)$ is in $L^0(\OO;C([0,t];E_0^*))$.
\item There is a mapping $C:\Omega\to \R_+$ and an $\e>0$ such that for all $s\in [0,t)$,
\[\| ((-A(s))^{1+\beta})^*\varphi(s)\| + \|((-A(s))^{\beta})^*\varphi'(s)\|\leq C(t-s)^{-1 + \varepsilon}.\]
\end{enumerate}
\end{definition}

\begin{example}\label{ex:rmrk-Gamma_t}
Let $x^*\in D(A(t)^*)$. For all $\beta\in [0,\mu^*+\nu^*-1)$ the process $\varphi: [0,t]\times \OO\to E_0^*$ defined by $\varphi(s) = S(t,s)^*x^*$ belongs to $\Gamma_{t,\beta}$. Indeed, first of all $\varphi\in L^0(\Omega;C^1([0,t];E_0^*))$ (see below \eqref{eqn:adjoint-evol-drvtv}). Moreover,  $-A(s)^* \varphi(s) =  \varphi'(s) =  -A(s)^*S(t,s)x^*$ is continuous,
and by the adjoint version of \eqref{thrm-evolsys-frct-ineq5} the latter satisfies
\begin{align*}
\|((-A(s))^{1+\beta})^*S(t,s)^*x^*\|& \leq \|((-A(s))^{1+\beta})^*S(t,s)^*(A(t)^{-\lambda})^*\| \|((-A(t))^{\lambda})^* x^*\|
\\ & \leq C(t-s)^{-1-\beta+\lambda}  \|((-A(t))^{\lambda})^* x^*\|,
\end{align*}
for all $\lambda\in (0, 1)$. The latter satisfies the required condition whenever $\lambda\in (\beta,1)$.
\end{example}

In the next theorem we show the equivalence of the formulas \eqref{eq:reprform} and \eqref{eq:var1}. It extends Proposition \ref{prop:equivstrongrepr} to the unbounded setting.

\begin{theorem}\label{thm:varrepr}
Let $p\in (2, \infty)$ and let $G$ be an adapted process belonging to $L^0(\OO;L^p(0,T;\gamma(H,\tilde{E}_{-\theta})))$.
\begin{enumerate}
\item If for all $t\in [0,T]$, \eqref{eq:reprform} holds a.s., then for all $\beta\in (\theta, \eta_-)$, for all $t\in [0,T]$ and for all $\varphi\in \Gamma_{t,\beta}$ the identity \eqref{eq:var1} holds a.s.
\item If $U\in L^0(\Omega;L^1(0,T;E_0))$ and there is $\beta\in (\theta, \eta_-)$ such that for all $t\in [0,T]$, for all $\varphi\in \Gamma_{t,\beta}$ the identity \eqref{eq:var1} holds a.s., then for all $t\in [0,T]$, $U$ satisfies \eqref{eq:reprform} a.s.
\end{enumerate}
\end{theorem}

We have already seen that \eqref{eq:reprform} is well-defined. Also all terms in \eqref{eq:var1}  are well-defined. For instance
\begin{align*}
|\lb I(\one_{(0,s)}G),\varphi'(s)\rb| &= |\lb (-A(s))^{-\beta} I(\one_{(0,s)}G), ((-A(s))^{\beta})^* \varphi'(s)\rb|
\\ & \leq C \sup_{r\in [0,T]}\|I(\one_{(0,r)}G)\|_{E_{-\theta}}  (t-s)^{-1+\varepsilon}
\end{align*}
and the latter is integrable with respect to $s\in (0,t)$.

In the proof we will use  the following well-known identity:
\begin{equation}\label{eq:primitieve}
\int_0^s \lb S(s,r) A(r) x,x^*\rb \, dr =  \lb S(s,0) x,x^*\rb  -
\lb x,x^*\rb, \ \ \ x\in \tilde{E}_{-\theta}, x^*\in D((-A(s)^*)^{\beta}),
\end{equation}
where $\beta\in (\theta, \eta_-)$. The identity \eqref{eq:primitieve} can be proved by first taking $x^* \in D(A(s)^*)$ and $x\in E_0$, and an approximation argument to obtain the identity for $x\in \tilde{E}_{-\theta}$ and $x^*\in D(((-A(s))^{\beta})^*)$.

\begin{proof}
(1): \ Assume \eqref{eq:reprform} holds and fix $s\in [0,T]$ for the moment.
Let $\beta\in (\theta, \eta_-)$ and choose $\lambda\in (\beta, \eta_-)$. Let $x^*\in D(((-A(s))^{\lambda})^*)$.
By \eqref{thrm-evolsys-frct-ineq5},
\[\|(-A(s))^{-\lambda} S(s,r) (-A(r))^{1+\beta}\|_{\calL(E_0)} \leq C(s-r)^{-1-\beta+\lambda}.\]
Since also $r\mapsto I(\one_{(0,r)}G)$ is in $L^0(\Omega;L^\infty(0,T;\tilde{E}_{-\theta}))$ it follows that
\begin{align*}
&\int_0^s  |\lb S(s,r) A(r) I(\one_{(r,s)}G),x^* \rb| \, dr  \\ & =
\int_0^s |\lb  (-A(s))^{-\lambda} S(s,r) (-A(r))^{1+\beta} (-A(r))^{-\beta} I(\one_{(r,s)}G), ((-A(s))^{\lambda})^* x^*  \rb| \, dr
\\ & \leq C_G \int_0^s (s-r)^{-1-\beta+\lambda} \|((-A(s))^{\lambda})^* x^* \| \, dr
\\ & \leq C_G' \|((-A(s))^{\lambda})^* x^*\|.
\end{align*}
Since $I(\one_{(r,s)}G) = I(\one_{(0,s)}G) - I(\one_{(0,r)}G)$, we can write
\begin{align*}
\int_0^s & \lb S(s,r) A(r) I(\one_{(r,s)}G),x^*\rb \, dr \\ & =
\int_0^s \lb S(s,r) A(r) I(\one_{(0,s)}G),x^* \rb \, dr - \int_0^s  \lb S(s,r) A(r) I(\one_{(0,r)}G),x^* \rb \, dr.
\end{align*}

By \eqref{eq:primitieve} (with $x= I(\one_{(0,s)}G)$), \eqref{eq:reprform} and the above identity we find that
\begin{align}\label{eq:weakidhulp}
\lb U(s), x^*\rb =  \int_0^s  \lb S(s,r) A(r) I(\one_{(0,r)}G),x^* \rb \, dr +  \lb I(\one_{(0,s)}G),x^*\rb.
\end{align}
Now let $\varphi\in \Gamma_{t,\beta}$. Applying the above with $x^* = A(s)^*\varphi(s)$
and integrating over $s\in (0,t)$ we find that
\begin{equation}\label{eq:calcvar}
\begin{aligned}
\int_0^t & \lb U(s), A(s)^*\varphi(s)\rb\, ds  - \int_0^t \lb I(\one_{(0,s)}G),A(s)^*\varphi(s)\rb\, ds
\\ & =  \int_0^t \int_0^s  \lb S(s,r) A(r) I(\one_{(0,r)}G),A(s)^*\varphi(s) \rb \, dr\, ds
\\ & = \int_0^t \int_r^t  \lb S(s,r) A(r) I(\one_{(0,r)}G),A(s)^*\varphi(s) \rb \, ds\, dr.
\end{aligned}
\end{equation}
Since $\frac{d}{dt}S(t,s) = A(t)S(t,s)$, with an approximation argument it follows that for all $x\in \tilde{E}_{-\theta}$ and $0\leq r\leq t\leq T$,
\begin{equation}\label{eqn:equiv-reprsol-varsol-ibpfrmla}
\begin{aligned}
\lb S(t,r)A(r) x, \varphi(t)\rb - \lb x, A(r)^*\varphi(r) \rb &= \int_r^t \lb S(s,r)A(r)x, A(s)^*\varphi(s)\rb\;ds\\ & \qquad  + \int_r^t \lb S(s,r)A(r)x, \varphi'(s)\rb\;ds.
\end{aligned}
\end{equation}
Note that the above integrals converge absolutely. Indeed, for all $\varepsilon>0$ small, one has by \eqref{thrm-evolsys-frct-ineq5} and the assumption on $\varphi$ that
\begin{align*}
|\lb S(s,r)A(r) x, A(s)^*\varphi(s)\rb|& = |\lb (-A(s))^{-\lambda}S(s,r)A(r) x, ((-A(s))^{1+\lambda})^*\varphi(s)\rb|
\\ & \leq C(s-r)^{-1-\theta-\varepsilon+\lambda} (t-s)^{-1+\varepsilon}.
\end{align*}
The latter is clearly integrable with respect to $s\in (r,t)$ for $\varepsilon>0$ small enough. The same estimate holds with $A(s)^*\varphi(s)$ replaced by $\varphi'(s)$.

Using \eqref{eqn:equiv-reprsol-varsol-ibpfrmla} in the identity \eqref{eq:calcvar} we find that
\begin{align*}
\int_0^t \lb U(s), A(s)^*\varphi(s)\rb\, ds & = \int_0^t \lb S(t,r) A(r) I(\one_{(0,r)}G),\varphi(t) \rb\, dr
\\ & \qquad - \int_0^t \int_r^t \lb S(s,r) A(r) I(\one_{(0,r)}G),\varphi'(s) \rb \, ds\, dr.
\end{align*}
Therefore, by \eqref{eq:weakidhulp} applied with $s=t$ and $x^* = \varphi(t)$, and Fubini's theorem we find
\begin{align*}
\int_0^t \lb U(s), A(s)^*\varphi(s)\rb\, ds & = \lb U(t), \varphi(t)\rb - \lb I(\one_{(0,t)}G),\varphi(t)\rb \\ & \qquad -\int_0^t \lb U(s),\varphi'(s)\rb \,ds + \int_0^t \lb I(\one_{(0,s)}G),\varphi'(s)\rb \, ds.
\end{align*}
This implies that $U$ satisfies \eqref{eq:var1}.

(2): Assume \eqref{eq:var1} holds. Fix $t\in [0,T]$ and $x^*\in D(A(t)^*)$. By Example \ref{ex:rmrk-Gamma_t} the process $\varphi:[0,t]\times\Omega\to E_0^*$ given by $\varphi(s) = S(t,s)^*x^*$ is in $\Gamma_{t,\beta}$ for all $\beta\in [0,\eta_-)$. Applying \eqref{eq:var1} and using that $\varphi'(s) = -A(s)^*\varphi(s)$ we find that
\begin{align*}
\lb U(t), x^*\rb & =\int_0^t \lb S(t,s)A(s)I(\one_{(0,s)}G),x^*\rb \;ds + \lb I(\one_{(0,t)}G), x^*\rb
\end{align*}
and as in part (1) of the proof this can be rewritten as
\begin{equation}\label{eq:almostrepr}
\lb U(t), x^*\rb  =-\int_0^t \lb S(t,s)A(s)I(\one_{(s,t)}G),x^*\rb \;ds + \lb S(t,0) I(\one_{(0,t)}G), x^*\rb.
\end{equation}
The identity \eqref{eq:reprform} follows from the Hahn-Banach theorem and density of $D(A(t)^*)$ in $E_0^*$.
\end{proof}

\begin{remark}
If $\varphi$ in \eqref{eq:var1} is not dependent of $\Omega$, then the stochastic Fubini theorem and integration by parts show that \eqref{eq:var1} is equivalent with
\begin{equation}\label{eq:var2}
\begin{aligned}
\lb U(t), \varphi(t)\rb & = \int_0^t \lb U(s), A(s)^*\varphi(s)\rb \;ds + \int_0^t \lb U(s),\varphi'(s)\rb \;ds \\ & \qquad - \int_0^t G(s)^* \varphi(s)\, d W(s).
\end{aligned}
\end{equation}
This solution concept coincides with the one in \cite{Veraar-SEE} and is usually referred to as a variational solution. Using the forward integral one can obtain \eqref{eq:var2} from \eqref{eq:var1} for $\varphi$ depending on $\omega$ in a nonadapted way (see Section \ref{subsec:mild} for the definition of the forward integral).
\end{remark}

In the next theorem we show the equivalence of the pathwise mild solution \eqref{eq:reprform} and the usual weak formulation \eqref{eq:weak1}. It extends also Proposition \ref{prop:equivstrongrepr} to the unbounded setting.
\begin{theorem}\label{thm:weakrepr}
Let $G$ be an adapted process in $L^0(\OO;L^p(0,T;\gamma(H,\tilde{E}_{-\theta})))$, and let
\[F = \bigcap_{t\in [0,T], \omega\in\Omega} D((A(t,\omega))^*).\]
\begin{enumerate}[(1)]
\item Assume $D(A(t))=D(A(0))$ isomorphically with uniform estimates in $t\in [0,T]$ and $\omega\in \Omega$. If for all $t\in [0,T]$, \eqref{eq:reprform} holds a.s., then for all $x^*\in F$, for all $t\in [0,T]$, \eqref{eq:weak1} holds a.s.

\item Assume $D(A(t)^*)=D(A(0)^*)$ isomorphically with uniform estimates in $t\in [0,T]$ and $\omega\in \Omega$. If for all $x^*\in F$ and $t\in [0,T]$, \eqref{eq:weak1} holds a.s., then for all $t\in [0,T]$, \eqref{eq:reprform} holds a.s.
\end{enumerate}
\end{theorem}

\begin{proof}
(1): \ First consider the case where $G(t)\in D(A(t))$ for all $t\in [0,T]$, and the process $t\mapsto A(t)G(t)$ is adapted in $L^0(\OO;L^p(0,T;\gamma(H,E_0))$. Let $x^*\in F$ and take $\varphi \equiv x^*$. Unfortunately,
$\varphi$ is not in $\Gamma_{t,0}$. However, due to the extra regularity of $G$, one can still proceed as in the proof of Theorem \ref{thm:varrepr} (1). Indeed, the only modification needed is that \eqref{eqn:equiv-reprsol-varsol-ibpfrmla} holds for $x^*\in F$ and $x\in E_1$. Moreover, since $\varphi'=0$,
\eqref{eq:weak1} follows from \eqref{eq:var1}.

Now let $G\in L^0(\OO;L^p(0,T;\gamma(H,\tilde{E}_{-\theta}))$ and define an approximation by $G_n(t) = n^2 R(n,A(t))^{-2} G(t)$. Let $U_n$ be given by \eqref{eq:reprform} with $G$ replaced by $G_n$. Then by the above, $U_n$ satisfies
\begin{align}\label{eq:weak1hulp}
\lb U_n(t), x^*\rb &=
\int_0^t \lb U_n(s), A(s)^* x^*\rb \;ds + \int_0^t G_n(s)^*x^* \;dW(s).
\end{align}
By the dominated convergence theorem, almost surely $G_n\to G$ in the space $L^p(0,T;\gamma(H,\tilde{E}_{-\theta}))$. Therefore, Proposition \ref{prop:type2prop} and Theorem \ref{thm:pathregularity} yield that $U_n\to U$ in $L^0(\Omega;L^p(0,T;E_0))$. Letting $n\to \infty$ in \eqref{eq:weak1hulp}, we obtain \eqref{eq:weak1}.

(2): \ The strategy of the proof is to show that $U$ satisfies \eqref{eq:var1} and to apply Theorem \ref{thm:varrepr}. In order to show \eqref{eq:var1} we need to allow the functional $x^*\in F$ to be dependent on $s\in [0,t]$ and $\omega\in \Omega$. In order to do so, fix $t\in [0,T]$, let $f\in C^1([0,t])$ and $x^*\in F$. Let $\varphi = f \otimes x^*$. By integration by parts and \eqref{eq:weak1}  (applied twice) we obtain
\begin{align*}
\lb & U(t), \varphi(t)\rb -\lb I(\one_{[0,t]} G), \varphi(t)\rb  = \int_0^t \lb U(s), A(s)^*x^*\rb \, ds \, f(t)
\\ & = \int_0^t \lb U(s), A(s)^*\varphi(s)\rb \, ds  + \int_0^t \int_0^s \lb U(r), A(r)^* x^*\rb \, dr \, f'(s)\, ds
\\ & = \int_0^t \lb U(s), A(s)^*\varphi(s)\rb \, ds  + \int_0^t \lb U(s), x^*\rb \, f'(s)\, ds - \int_0^t \lb I(\one_{[0,s]}G), x^*\rb \, f'(s)\, ds
\\ & = \int_0^t \lb U(s), A(s)^*\varphi(s)\rb \, ds  + \int_0^t \lb U(s), \varphi'(s)\rb \, ds - \int_0^t \lb I(\one_{[0,s]}G), \varphi'(s)\rb \, ds.
\end{align*}
This yields \eqref{eq:var1} for the special $\varphi$ as above. By linearity and approximation the identity \eqref{eq:var1} can be extended to all $\varphi\in C^1([0,t];E_0^*)\cap C([0,t];F)$. Clearly, the identity extends
to simple functions $\varphi:\Omega\to C^1([0,t];E_0^*)\cap C([0,t];F)$ and by approximation it extends to any $\varphi\in L^0(\Omega;C^1([0,t];E_0^*)\cap C([0,t];F))$. Now let $x^*\in F$ be arbitrary and let $\varphi(s) = S(t,s)^*x^*$. Then as in the proof of Theorem \ref{thm:varrepr} (2) we see that $\varphi\in L^0(\Omega;C^1([0,t];E_0^*)\cap C([0,t];F))$ and \eqref{eq:almostrepr} follows.
\end{proof}

\subsection{Forward integration and mild solutions\label{subsec:mild}}
In this section we show how the forward integral can be used to define the mild solution to \eqref{eq:SEEsimple} and show that it coincides with the pathwise mild solution \eqref{eq:reprform}. The forward integral was developed by Russo and Vallois in \cite{RV91}, \cite{RV93}, and can be used to integrate nonadapted integrands and is based on a regularization procedure. We refer to \cite{RV-survey} for a survey on the subject and a detailed collection of references.

Fix an orthonormal basis $(h_k)_{k\geq 1}$ for $H$. For $G\in L^0(\OO;L^2(0,T;\g(H,E_0)))$ define the sequence $(I^-(G,n))_{n=1}^\infty$ by
\begin{align*}
I^-(G,n) = \sum_{k=1}^n n\int_0^T G(s) h_k (W(s+1/n)h_k - W(s)h_k)\;ds.
\end{align*}
The process $G$ is called {\em forward integrable} if $(I^-(G,n))_{n\geq 1}$ converges in probability.
In that case, the limit is called {\em the forward integral} of $G$ and is denoted by
\[I^{-}(G) = \int_0^T G\;dW^- = \int_0^T G(s)\;dW^-(s).\]
This definition is less general than the one in \cite{PV-forward}, but will suffice for our purposes here.

In \cite{PV-forward} it has been shown that for {\umd} Banach spaces the forward integral extends the It\^o integral from \cite{NVW1}. In particular, the forward integral as defined above extends the stochastic integral as described in Section \ref{subsec:stochint}.

We will now show that the forward integral can be used to extend the concept of mild solutions to the case where $A(t)$ is random. The proof will be based on a pointwise multiplier result for the forward integral from \cite{PV-forward}.

\begin{theorem}\label{thm:forwardequiv}
Assume \ref{asmp:A-mble-adptd}-\ref{hpths:interpol-spce-wkextra}. Let $p\in (2, \infty)$.
Let $\theta \in [0,\tfrac12\wedge \eta_-)$.
Assume $\delta < \min\{\frac12 - \theta - \tfrac1p, \eta_+\}$.
Let $G$ be an adapted process in $L^0(\OO;L^p(0,T;\g(H,\tilde{E}_{-\theta})))$. For every $t\in [0,T]$, the process $s\mapsto S(t,s)G(s)$ is forward integrable on $[0,t]$ with values in $\tilde{E}_\delta$, and
\begin{align}\label{eqn:weaksolrepr-approxseq}
U(t) = \int_0^t S(t,s)G(s)\;dW^-(s),
\end{align}
where $U$ is given by \eqref{eq:reprform}.
\end{theorem}
The above identity is mainly of theoretical interest as it is rather difficult to prove estimates for the forward integral in a direct way. Of course \eqref{eq:reprform} allows to obtain such estimates. Due to \eqref{eqn:weaksolrepr-approxseq} one could call $U$ a {\em forward mild solution} to \eqref{eq:SEEsimple}.

As a consequence of Theorems \ref{thm:weakrepr} and \ref{thm:forwardequiv}, there is an equivalence between weak solutions and forward mild solutions. Under different assumptions it was shown in \cite[Proposition 5.3]{NualartLeon} that every forward mild solution is a weak solution.

\begin{proof}
Define $M:[0,t]\times\OO\to \calL(\tilde{E}_{-\theta}, \tilde{E}_{\delta})$  by $M(s) = S(t,s)$. Let $N:[0,t)\times \OO\to \calL(\tilde{E}_{-\theta}, \tilde{E}_{\delta})$ be given by $N(s) = -S(t,s)A(s)$.
Then by Lemma \ref{lem:diff-prop-evolsys}, $M(s) = M(0) + \int_0^s N(r) \, dr$ for $s\in [0,t)$ and thus $M$ is continuously differentiable with derivative $N$.
By Remark \ref{rmrk:evol-ineqs} there is a mapping $C:\OO\to \R_+$ such that
\[\|N(s)\|_{\calL(\tilde{E}_{-\theta}, \tilde{E}_{\delta})}\leq C(t-s)^{-1-\delta-\theta}.\]
Now by the non-adapted multiplier result for the forward integral from \cite{PV-forward} we find that $M G$ is forward integrable and
\begin{align*}
\int_0^t S(t,s)G(s)\;dW^-(s) &= \int_0^t M(s) G(s)\, d W^{-}(s) \\ & = M(0) I(G) + \int_0^t N(s) I(\one_{[s,t]}G) \, ds = U(t).
\end{align*}
\end{proof}

\begin{remark}
Another pathwise approach to stochastic evolution equations can be given using a Wong-Zakai type regularization procedure. In \cite{AT84}, this procedure has been considered in a linear setting with deterministic $A(t)$. In \cite{TZ06}, in the case that $A$ does not dependent on time, existence and uniqueness of martingale solutions for general stochastic evolution equations have been proved using Wong-Zakai regularization.
\end{remark}

\section{Semilinear stochastic evolution equations\label{sec:semil}}
In this section we assume Hypotheses \ref{asmp:A-mble-adptd}-\ref{hpths:interpol-spce-wkextra}.
We will apply the results of the previous sections to study the following stochastic evolution equation on the Banach space $E_0$
\begin{align}\label{defn:SEE2}
\left\{
    \begin{array}{ll}
        dU(t) &= (A(t)U(t) + F(t,U(t)))\;dt + B(t, U(t))\;dW(t), \\
        U(0) &= u_0.
    \end{array}
\right.
\end{align}
Here $F$ and $B$ will be suitable nonlinearities of semilinear type.
In Section \ref{subs:solcon} we will first state the main hypotheses on $F$ and $B$ and define the concept of a pathwise mild solution. In Section \ref{subs:uniformO} we will prove that there is a unique pathwise mild solution under the additional assumption that the constants in the (AT)-conditions do not depend on $\omega$.
The uniformity condition \ref{hpths:unifrm-const} will be removed in Section \ref{subs:generalO} by localizing the random drift $A$.

\subsection{Setting and solution concepts\label{subs:solcon}}

Recall that the spaces $\tilde{E}_{\eta}$ were defined in \ref{hpths:interpol-spce-wk} in Section \ref{sec:pathwisereg}.
We impose the following assumptions on $F$ and $B$ throughout this section:

\let\ALTERWERTA\theenumi
\let\ALTERWERTB\labelenumi
\def\theenumi{(HF)}
\def\labelenumi{(HF)}
\begin{enumerate}
  \item\label{asmp:Lip-asmpt-F}
  Let $a \in [0,\eta_+)$ and $\theta_F \in [0,\eta_-)$ be such that $a + \theta_F < 1$. For all $x\in \tilde{E}_a$, $(t,\omega) \mapsto F(t,\omega,x) \in \tilde{E}_{-\theta_F}$ is strongly measurable and adapted. Moreover, there exist constants $L_F$ and $C_F$ such that for all $t\in [0,T], \omega\in\Omega, x,y\in \tilde{E}_a$,
        \begin{align*}
            \|F(t,\omega, x) - F(t,\omega, y)\|_{\tilde{E}_{-\theta_F}} &\leq L_F\|x-y\|_{\tilde{E}_a}, \\
            \|F(t,\omega, x)\|_{\tilde{E}_{-\theta_F}} &\leq C_F(1+\|x\|_{\tilde{E}_a}).
        \end{align*}
\end{enumerate}
\let\theenumi\ALTERWERTA
\let\labelenumi\ALTERWERTB

\let\ALTERWERTA\theenumi
\let\ALTERWERTB\labelenumi
\def\theenumi{(HB)}
\def\labelenumi{(HB)}
\begin{enumerate}
  \item\label{asmp:Lip-asmpt-B} Let $a \in [0,\eta_+)$ and $\theta_B \in [0,\eta_-)$ be such that $a + \theta_B < 1/2$. For all $x\in \tilde{E}_a$, $(t,\omega) \mapsto B(t,\omega,x) \in \g(H,\tilde{E}_{-\theta_B})$ is strongly measurable and adapted. Moreover, there exist constants $L_B$ and $C_B$ such that for all $t\in [0,T], \omega\in\Omega, x,y\in \tilde{E}_a$,
        \begin{align*}
            \|B(t,\omega, x) - B(t,\omega, y)\|_{\g(H,\tilde{E}_{-\theta_B})} &\leq L_B\|x-y\|_{\tilde{E}_a}, \\
            \|B(t,\omega, x)\|_{\g(H,\tilde{E}_{-\theta_B})} &\leq C_B(1+\|x\|_{\tilde{E}_a}).
        \end{align*}
\end{enumerate}
\let\theenumi\ALTERWERTA
\let\labelenumi\ALTERWERTB

Let $p\in (2, \infty)$ and consider adapted processes $f\in L^0(\Omega;L^p(0,T;\tilde{E}_{-\theta_F}))$ and $G\in L^0(\Omega;L^p(0,T;\g(H,\tilde{E}_{-\theta_B})))$. In the sequel we will write
\begin{align*}
S*f(t) &:=\int_0^t S(t,s)f(s)\;ds,
\\ S\diamond G(t) &:= - \int_0^t S(t,s)A(s)I(\one_{(s,t)} G)\;ds + S(t,0)I(\one_{(0,t)} G),
\end{align*}
for the deterministic and stochastic (generalized) convolution.

The integral $\int_0^t S(t,s)A(s)I(\one_{(s,t)} G)\;ds$ was extensively studied in Section \ref{sec:repr}. Recall from Theorem \ref{thm:pathregularity} that is it well-defined and defines an adapted process in $L^0(\Omega;W^{\lambda,p}(0,T;\tilde{E}_{\delta}))$ for suitable $\lambda$ and $\delta$.

\begin{definition} \label{defn:repr-sol}
Let $2<p<\infty$. An adapted process $U\in L^0(\Omega;L^p(0,T;\tilde{E}_a))$ is called a {\em pathwise mild solution} of (\ref{defn:SEE2}) almost surely, for almost all $t\in [0,T]$,
\begin{equation}\label{eqn:reprsoldef}
\begin{aligned}
U(t) = S(t,0)u_0 & + S*F(\cdot, U)(t) + S\diamond B(\cdot, U)(t).
\end{aligned}
\end{equation}
\end{definition}
Note that the convolutions in \eqref{eqn:reprsoldef} might only be defined for almost all $t\in [0,T]$. However, if $p\in (2, \infty)$ is large enough, then they are defined in a pointwise sense and pathwise continuous (see Theorems \ref{thm:F-Holder} and \ref{thm:SP-Holder-ctu}) and we obtain that for almost all $\omega\in \Omega$, for all $t\in [0,T]$ \eqref{eqn:reprsoldef} holds.

One can extend Proposition \ref{prop:equivstrongrepr} and Theorems \ref{thm:varrepr}, \ref{thm:weakrepr} and \ref{thm:forwardequiv} to the nonlinear setting. Indeed, this follows by taking $G = B(\cdot, U)$ and including the terms $F$ and $u_0$. The latter two terms do not create any problems despite the randomness of $A$, because the terms are defined in a pathwise way, and therefore can be treated as in \cite{Veraar-SEE}. As a consequence we deduce that \eqref{eqn:reprsoldef} yields the ``right'' solution of \eqref{defn:SEE2} in many ways (variational, forward mild, weak).

\subsection{Results under a uniformity condition in $\OO$\label{subs:uniformO}}

In this section we additionally assume the following uniformity condition.
\let\ALTERWERTA\theenumi
\let\ALTERWERTB\labelenumi
\def\theenumi{(H5)}
\def\labelenumi{(H5)}
\begin{enumerate}
  \item\label{hpths:unifrm-const} The mapping $L:\Omega\to \R_+$ from \ref{asmp:AT2} for $A(t)$ and $A(t)^*$ is bounded in $\Omega$.
\end{enumerate}
\let\theenumi\ALTERWERTA
\let\labelenumi\ALTERWERTB

Under Hypothesis \ref{hpths:unifrm-const}, it is clear from the proofs that most of the constants in Sections \ref{eq:SEF} and \ref{sec:pathwisereg} become uniform in $\Omega$. In Section \ref{subs:generalO} we will show how to obtain well-posedness without the condition \ref{hpths:unifrm-const}.

For a Banach space $X$, we write $B([0,T];X)$ for the strongly measurable functions $f:[0,T]\to X$.
For $\delta\in (-1,\eta_+)$ and $p\in (2,\infty)$ let $Z_{\delta}^p$ be the subspace of strongly measurable adapted processes $u:[0,T]\times\Omega\to \tilde{E}_{\delta}$ for which
\[\|u\|_{Z_{\delta}^p}:=\sup_{t\in [0,T]} \|u(t)\|_{L^p(\Omega;\tilde{E}_{\delta})}<\infty.\]
Define the operator $L: Z^p_\delta\to Z^p_\delta$ by
\begin{align}\label{defn:oprtr-L}
(L(U))(t) = S(t,0)u_0 & + S*F(\cdot, U)(t) + S\diamond B(\cdot, U)(t).
\end{align}
In the next lemma we show that $L$ is well-defined and is a strict contraction in a suitable equivalent norm on $Z^p_a$.

\begin{lemma}\label{lemma:fixedpoint-arg}
Assume \ref{asmp:A-mble-adptd}--\ref{hpths:unifrm-const}, \ref{asmp:Lip-asmpt-F} and \ref{asmp:Lip-asmpt-B}. Let $p\in (2,\infty)$.
If the process $t\mapsto S(t,0) u_0$ is in $Z^p_a$, then
$L$ maps $Z^p_a$ into itself and there is an equivalent norm $\nn\cdot\nn$ on $Z^p_a$ such that for every $u, v\in Z^p_a$,
\begin{align}\label{ineq:fixedpoint-arg-contraction}
\nn L(u)-L(v)\nn_{Z^p_a} \leq \frac12 \nn u-v\nn_{Z^p_a}.
\end{align}
Moreover, there exists a constant $C$ independent of $u_0$ such that
\begin{align}\label{ineq:fixedpoint-arg-contraction2}
\nn L(u)\nn_{Z^p_a} \leq C+ \nn t \mapsto S(t,0) u_0\nn_{Z^p_a} +\frac12 \nn u\nn_{Z^p_a}.
\end{align}
\end{lemma}

\begin{proof}
Choose $\varepsilon>0$ so small that $a+\theta_F+\varepsilon<1$ and $a+\theta_B+\varepsilon<1/2$.
We will first prove several estimates for the individual parts of the mapping $L$. Note that the stochastic convolution term $S\diamond B(\cdot, U)$ is defined in a nonstandard way, and therefore we need to give all details in the proof below. Conclusions will be derived afterwards.

For $\kappa\geq 0$ arbitrary but fixed for the moment, define the equivalent norm on ${Z_\delta^p}$ by
\[\nn u\nn_{Z_\delta^p} = \sup_{t\in [0,T]} e^{-\kappa t} \|u(t)\|_{L^p(\Omega;\tilde{E}_\delta)}.\]
We also let
\[\nn G\nn_{Z^p_{\delta}(\gamma)} = \sup_{t\in [0,T]} e^{-\kappa t} \|G(t)\|_{L^p(\Omega;\gamma(H,\tilde{E}_{\delta}))}.\]

{\em Stochastic convolution:}
Let $G\in Z^p_{-\theta_B}$  be arbitrary. Clearly, we can write $\|S\diamond G(t)\|_{L^p(\Omega;\tilde{E}_a)}\leq T_1(t) + T_2(t)$,
where
\[T_1 (t):= \Big\|\int_0^t S(t,s)A(s)I(\one_{(s,t)} G)\;ds\Big\|_{L^p(\Omega;\tilde{E}_a)},\]
and $T_2(t)= \|S(t,0)I(\one_{(0,t)} G)\|_{L^p(\Omega;\tilde{E}_a)}$.
To estimate $T_1$ note that by Remark \ref{rmrk:evol-ineqs}
\begin{align*}
T_1(t)&\leq C\int_0^t (t-s)^{-1-a-\theta_B-\varepsilon}  \|I(\one_{(s,t)} G)\|_{L^p(\Omega;\tilde{E}_{-\theta_B})} \;ds.
\end{align*}
By \ref{hpths:unifrm-const}, $C$ is independent of $\omega$. By \eqref{eq:JGest} and Minkowski's inequality we have
\begin{equation}\label{eq:stochinkappaest}
\begin{aligned}
\|I(\one_{(s,t)} G)\|_{L^p(\Omega;\tilde{E}_{-\theta_B})} & \leq C\|G\|_{L^p(\Omega;L^2(s,t;\g(H,\tilde{E}_{-\theta_B})))}
\\ & \leq C\|G\|_{L^2(s,t;L^p(\Omega;\g(H,\tilde{E}_{-\theta_B})))}
\\ & \leq C\Big(\int_s^t e^{2\kappa\sigma} \, d\sigma\Big)^{1/2} \nn G\nn_{Z^p_{-\theta_B}(\gamma)}
\\ & = C\kappa^{-1/2} (e^{2\kappa t} - e^{2\kappa s})^{1/2} \nn G\nn_{Z^p_{-\theta_B}(\gamma)}.
\end{aligned}
\end{equation}
Therefore, we find that
\begin{align*}
\sup_{t\in[0,T]} &e^{-\kappa t}T_1(t) \\& \leq C \sup_{t\in[0,T]}\int_0^t (t-s)^{-1-a-\theta_B-\varepsilon} \big(\kappa^{-1}(1 - e^{-2\kappa(t- s)})\big)^{1/2} \, ds \nn G\nn_{Z^p_{-\theta_B}(\gamma)}
\\ & = C \sup_{t\in[0,T]}\int_0^t \sigma^{-1-a-\theta_B-\varepsilon}  \kappa^{-1/2} (1 - e^{-2\kappa \sigma})^{1/2} \, d\sigma
\nn G\nn_{Z^p_{-\theta_B}(\gamma)}
\\ & \leq C \phi_2(\kappa) \nn G\nn_{Z^p_{-\theta_B}(\gamma)},
\end{align*}
where $\phi_2$ is given by
\[\phi_2(\kappa) = \kappa^{-\frac12+a+\theta_B+\varepsilon} \int_0^\infty \sigma^{-1-a-\theta_B-\varepsilon}  (1 - e^{-2\sigma})^{1/2}\, d\sigma.\]
Since $a+\theta_B+\varepsilon<\frac12$, the latter is finite. Moreover, $\lim_{\kappa\to \infty}\phi_2(\kappa) = 0$.

To estimate $T_2(t)$ note that by Remark \ref{rmrk:evol-ineqs}, \ref{hpths:unifrm-const} and \eqref{eq:stochinkappaest} with $s=0$,
\begin{align*}
T_2(t) & \leq  C t^{-a-\theta_B-\varepsilon} \|I(\one_{(0,t)} G)\|_{L^p(\Omega;\tilde{E}_{-\theta_B})}
\\ & \leq C t^{-a-\theta_B-\varepsilon} \kappa^{-1/2} (e^{2\kappa t} -1)^{1/2} \nn G\nn_{Z^p_{-\theta_B}(\gamma)}.
\end{align*}
Therefore, using $\sup_{\sigma\geq 0} \sigma^{-a-\theta_B-\varepsilon} (1-e^{-2\sigma})^{1/2}<\infty$, we find that
\begin{align*}
\sup_{t\in[0,T]} e^{-\kappa t}T_2(t) \leq C \kappa^{-\frac12+a+\theta_B+\varepsilon} \nn G\nn_{Z^p_{-\theta_B}(\gamma)}.
\end{align*}
Combining the estimate for $T_1$ and $T_2$ we find that
\begin{equation}\label{eq:estGconv}
\nn S\diamond G\nn_{Z^p_a} \leq C \phi_3(\kappa) \nn G\nn_{Z^p_{-\theta_B}(\gamma)},
\end{equation}
where $\phi_3(\kappa)\to 0$ if $\kappa\to \infty$.

Now let $u,v\in Z^p_a$. By the hypothesis \ref{asmp:Lip-asmpt-B}, $B(\cdot, u)$ and $B(\cdot, v)$ are in $Z^p_{-\theta_B}$ and therefore, by the above we find that $S\diamond B(\cdot,u)$ and $S\diamond B(\cdot,v)$ are in $Z^p_a$ again. Moreover, applying \eqref{eq:estGconv} with $G = B(\cdot, u) - B(\cdot, v)$ it follows that
\begin{align*}
\nn S\diamond B(\cdot, u) - S\diamond B(\cdot, v)\nn_{Z^p_a}& \leq C \phi_3(\kappa) \nn B(\cdot, u) - B(\cdot, v)\nn_{Z^p_{-\theta_B}(\gamma)}\\ & \leq C\phi_3(\kappa) L_B \nn u-v\nn _{Z^p_a}.
\end{align*}

\textit{Deterministic convolution:} Let $f\in Z^p_{-\theta_F}$.
By Remark \ref{rmrk:evol-ineqs} applied pathwise and \ref{hpths:unifrm-const} one obtains
\begin{align*}
\|S * f(t)\|_{\tilde{E}_{a}}\leq C\int_0^t (t-\sigma)^{-a-\theta_F-\varepsilon} \|f(\sigma)\|_{\tilde{E}_{-\theta_F}}  \,  d\sigma, \ \ t\in [0,T]
\end{align*}
where $C$ is independent of $\omega$. In particular, taking $L^p(\Omega)$-norms on both sides we find that
\begin{align*}
\|S * f(t)\|_{L^p(\Omega;\tilde{E}_{a})}& \leq C\int_0^t (t-\sigma)^{-a-\theta_F-\varepsilon} \|f(\sigma)\|_{L^p(\Omega;\tilde{E}_{-\theta_F})}  \,  d\sigma.
\end{align*}
Using $e^{-\kappa t} = e^{-\kappa (t-\sigma)} e^{-\kappa \sigma}$, it follows that
\begin{equation}\label{eq:estfLp}
\begin{aligned}
\nn S * f\nn_{Z^p_a}& \leq C\nn f\nn_{Z^p_{-\theta_F}}\sup_{t\in [0,T]} \int_0^t e^{-\kappa(t-\sigma)}(t-\sigma)^{-a-\theta-\varepsilon} \,  d\sigma
\\&\leq C\phi_1(\kappa) \nn f\nn_{Z^p_{-\theta_F}},
\end{aligned}
\end{equation}
where
\[\phi_1(\kappa)=\int_0^\infty e^{-\kappa\sigma} \sigma^{-a-\theta-\varepsilon} \,  d\sigma.\]
Clearly, $\lim_{\kappa\to \infty}\phi_1(\kappa) = 0$.

Now let $u,v\in Z^p_a$. By the hypothesis \ref{asmp:Lip-asmpt-F}, $F(\cdot, u)$ and $F(\cdot, v)$ are in $Z^p_{-\theta_F}$ and therefore, we find that  $S*F(\cdot,u)$ and $S*F(\cdot,v)$ are in $Z^p_a$ again. Moreover, applying \eqref{eq:estfLp} with $f = F(\cdot, u) - F(\cdot, v)$ it follows that
\begin{align*}
\|S*F(\cdot, u) - S*F(\cdot, v)\|_{Z^p_a}& \leq
C\phi_1(\kappa) \nn F(\cdot, u) - F(\cdot, v)\nn_{Z^p_{-\theta_F}}  \,  d\sigma
\\ & \leq C \phi_1(\kappa) L_F \nn u-v\nn_{Z^p_a}.
\end{align*}

\textit{Conclusion.}
From the above computations, it follows that $L$ is a bounded operator on $Z^p_a$. Moreover, for all $u,v \in Z^p_a$,
\[ \nn L(u)-L(v) \nn_{Z^p_a} \leq C(L_F \phi_1(\kappa) + L_B \phi_3(\kappa)) \nn u-v\nn_{Z^p_a}.\]
Choosing $\kappa$ large enough, the result follows. Also, \eqref{ineq:fixedpoint-arg-contraction2} follows by taking $v \equiv 0$.
\end{proof}

As a consequence we obtain the following result.
\begin{theorem}\label{thm:existenceL^pcond}
Assume \ref{asmp:A-mble-adptd}--\ref{hpths:unifrm-const}, \ref{asmp:Lip-asmpt-F} and \ref{asmp:Lip-asmpt-B}. Let $p\in (2,\infty)$.
Let $\delta, \lambda > 0$ be such that $a+\delta + \lambda < \min\{\tfrac12 - \theta_B, 1-\theta_F,\eta_+\}$. If the process $t\mapsto S(t,0) u_0$ is in $Z^p_a$, then there exists a unique pathwise mild solution $U\in Z^p_a$  of \eqref{defn:SEE2}. Moreover, $U - S(t,0) u_0\in L^p(\Omega;W^{\lambda,p}(0,T;\tilde{E}_\delta))$ and there is a constant independent of $u_0$ such that
\begin{align}\label{eq:regest}
\|U-S(t,0) u_0\|_{L^p(\Omega;W^{\lambda,p}(0,T;\tilde{E}_{a+\delta}))} \leq C(1+ \|t \mapsto S(t,0) u_0\|_{Z^p_a}).
\end{align}
\end{theorem}
Of course by Sobolev embedding \eqref{eq:fractionalSobolevembReprForm} one can further deduce H\"older regularity of the solution.

\begin{proof}
By Lemma \ref{lemma:fixedpoint-arg} there exists a unique fixed point $U\in Z^p_a$ of $L$. Clearly, this implies that $U$ is a pathwise mild solution of \eqref{defn:SEE2}. Moreover, from \eqref{ineq:fixedpoint-arg-contraction2} we deduce that
\[\nn U\nn_{Z^p_a}\leq 2C+ 2\nn t \mapsto S(t,0) u_0\nn_{Z^p_a}.\]
Next we prove the regularity assertion. From Theorem \ref{thm:F-Holder}, \ref{asmp:Lip-asmpt-F} and the previous estimate we see that:
\begin{align*}
\|S*F(\cdot,U)\|_{L^p(\Omega;W^{\lambda,p}(0,T;\tilde{E}_{a+\delta}))} & \leq C\|F(\cdot,U)\|_{L^p((0,T)\times\Omega;\tilde{E}_{-\theta_F})}
\\ & \leq C(1+\|U\|_{L^p((0,T)\times\Omega;\tilde{E}_{a})})
\\ & \leq C(1+ \nn U\nn_{Z^p_a})\leq C (1+\nn t \mapsto S(t,0) u_0\nn_{Z^p_a}).
\end{align*}
Similarly, by Theorem \ref{thm:pathregularity} (with $\alpha\in (a+\delta+\theta_B+\lambda,\tfrac12)$), Proposition \ref{prop:type2prop} and \ref{asmp:Lip-asmpt-B}
\begin{align*}
\|S\diamond B(\cdot,U)\|_{L^p(\Omega;W^{\lambda,p}(0,T;\tilde{E}_{a+\delta}))} & \leq C\|J(B(\cdot,U))\|_{L^p(\Omega;W^{\alpha,p}(0,T;\tilde{E}_{-\theta_B}))}
\\ & \leq C\|B(\cdot,U)\|_{L^p(\Omega\times(0,T);\gamma(H,\tilde{E}_{-\theta_B}))}
\\ & \leq C(1+\|U\|_{L^p((0,T)\times\Omega;\tilde{E}_{a})})
\\ & \leq C(1+ \nn U\nn_{Z^p_a})\leq C (1+\nn t \mapsto S(t,0) u_0\nn_{Z^p_a}).
\end{align*}
Now \eqref{eq:regest} follows since $U-S(t,0) u_0 = S*F(\cdot,U) + S\diamond B(\cdot,U)$.
\end{proof}

One can extend the above existence and uniqueness result to the situation where $u_0: \Omega \to \tilde{E}_a$ is merely $\mathscr{F}_0$-measurable. For that, we will continue with a local uniqueness property that will be used frequently.
\begin{lemma}\label{lemma:as-equality}
Assume \ref{asmp:A-mble-adptd}--\ref{hpths:unifrm-const}, \ref{asmp:Lip-asmpt-F} and \ref{asmp:Lip-asmpt-B}. Let $\tilde{A}$ be another operator satisfying \ref{asmp:A-mble-adptd}, \ref{hpths:ATplus}, \ref{hpths:interpol-spce-wk} and \ref{hpths:unifrm-const} with the same spaces $(\tilde{E}_{\eta})_{-\eta_0<\eta\leq\eta_+}$ and let the evolution family generated by $\tilde{A}$ be denoted by $(\tilde{S}(t,s))_{0\leq s\leq t\leq T}$. Let $u_0,\tilde{u}_0:\Omega\to E_{a,1}^0$ be $\F_0$-measurable and such that $S(t,0)u_0, \tilde{S}(t,0)\tilde{u}_0\in Z^p_a$. Let $\tilde{L}$ be defined as $L$, but with $({S}(t,s))_{0\leq s\leq t\leq T}$ and $u_0$ replaced by  $(\tilde{S}(t,s))_{0\leq s\leq t\leq T}$  and $\tilde{u}_0$, respectively. Let $\Gamma \in \F_0$ and let $\tau$ be a stopping time.

Suppose for almost all $\omega\in \Gamma$ and all $t\in [0,\tau(\omega)]$, $A(t,\omega) = \tilde{A}(t,\omega)$ and $u_0(\omega) = \tilde{u}_0(\omega)$.
If $U,\tilde{U}\in Z^p_a$ are such that
\[\one_{\Gamma}\one_{[0,\tau]}  U = \one_{\Gamma}\one_{[0,\tau]}  L(U), \ \ \text{and} \ \ \one_{\Gamma}\one_{[0,\tau]} \tilde{U} = \one_{\Gamma}\one_{[0,\tau]} \tilde{L}(\tilde{U}),\]
then for almost all $\omega\in \Gamma$ and all $t\in [0,\tau(\omega)]$ one has $U(t) = \tilde{U}(t)$.
\end{lemma}

\begin{proof}
First we claim that for all $u\in Z^p_a$ one has
\begin{equation}\label{eq:identityL}
\one_{\Gamma}\one_{[0,\tau]} L(u) = \one_{\Gamma}\one_{[0,\tau]} L(v) = \one_{\Gamma}\one_{[0,\tau]} \tilde{L}(v) = \one_{\Gamma}\one_{[0,\tau]} \tilde{L}(u),
\end{equation}
where $v = \one_{\Gamma}\one_{[0,\tau]} u$. Indeed, by the (pathwise) uniqueness of the evolution family one has almost surely on $\Gamma$ for all $0\leq s\leq t\leq \tau$, $S(t,s) = \tilde{S}(t,s)$. Now the identity \eqref{eq:identityL} can be verified for each of the terms in $L$ and $\tilde{L}$.

For instance for the first part of stochastic convolution term one has
\begin{equation}\label{eq:stochlocal}
\begin{aligned}
\one_{\Gamma}\one_{[0,\tau]}(t)  & \int_0^t S(t,s) A(s) I(\one_{[s,t]}B(\cdot,u)) \, ds  \\ & = \one_{\Gamma}\one_{[0,\tau]}(t)  \int_0^t \one_{\Gamma}\one_{s\leq t\leq \tau}  S(t,s) A(s) I(\one_{[s,t]}B(\cdot,u)) \, ds.
\end{aligned}
\end{equation}
Now $\one_{\Gamma}\one_{s\leq t\leq \tau} S(t,s) = \one_{\Gamma}\one_{s\leq t\leq \tau} \tilde{S}(t,s)$, so we can replace $S$ by $\tilde{S}$ on the right-hand side of \eqref{eq:stochlocal}. Moreover, using a property of the forward integral \cite[Lemma 3.3]{PV-forward} (or the local property of the stochastic integral) one sees
\begin{align*}
\one_{\Gamma}\one_{s\leq t\leq \tau} I(\one_{[s,t]}B(\cdot,u)) & = I^{-}(\one_{\Gamma}\one_{s\leq t\leq \tau}  \one_{[s,t]}B(\cdot,u))
\\ & = I^{-}(\one_{\Gamma}\one_{s\leq t\leq \tau}  \one_{[s,t]}B(\cdot,v))
\\ & = \one_{\Gamma} \one_{[0,\tau]}(s)  I(\one_{[s,t]}B(\cdot,v)).
\end{align*}
Thus we can replace $u$ by $v$ on the right-hand side of \eqref{eq:stochlocal}.

We will now show how the statement of the lemma follows. Writing $V = \one_{\Gamma}\one_{[0,\tau]} U$ and $\tilde{V} = \one_{\Gamma}\one_{[0,\tau]} \tilde{U}$, it follows from the assumption, \eqref{eq:identityL} and Lemma \ref{lemma:fixedpoint-arg} that
\begin{align*}
\nn V- \tilde{V}\nn_{Z^p_a} &= \nn \one_{\Gamma}\one_{[0,\tau]}(L(U) - \tilde{L}(\tilde{U}))\nn_{Z^p_a}
\\  &= \nn \one_{\Gamma}\one_{[0,\tau]} (L(V) - L(\tilde{V}))\nn_{Z^p_a}\\
  &\leq \nn L(V) - L(\tilde{V})\nn_{Z^p_a}\\
 &\leq \frac12\nn V- \tilde{V}\nn_{Z^p_a}.
\end{align*}
Therefore, $V =  \tilde{V}$ in $Z^p_a$. Since by Theorem \ref{thm:existenceL^pcond} and Sobolev embedding, $U - S(\cdot,0)u_0$ and $\tilde{U}- \tilde{S}(\cdot,0)\tilde{u}_0$ have continuous paths, it follows that a.s.\ for all $t\in [0,T]$, $V(t) = \tilde{V}(t)$. This implies the required result.
\end{proof}

We combine the previous lemma with a localization argument to obtain the main result of this subsection.
\begin{theorem}\label{thrm:exst-mild-sol:initialcondgeneral, smalltime}
Assume \ref{asmp:A-mble-adptd}--\ref{hpths:unifrm-const}, \ref{asmp:Lip-asmpt-F} and \ref{asmp:Lip-asmpt-B}.
Let $\delta, \lambda > 0$ be such that $a+\delta + \lambda < \min\{\tfrac12 - \theta_B, 1-\theta_F,\eta_+\}$.
If $u_0:\Omega\to E_0$ is $\F_0$-measurable and $u_0\in E_{a,1}^0$ a.s., then the following holds:
\begin{enumerate}
\item There exists a unique adapted pathwise mild solution $U$ of \eqref{defn:SEE2} that belongs to $L^0(\Omega;C([0,T];\tilde{E}_a))$ of \eqref{defn:SEE2}. Moreover, $U- S(t,0) u_0$ belongs to $L^0(\Omega;C^{\lambda}(0,T;\tilde{E}_{a+\delta}))$.
\item If also $u_0\in E_{a+\beta,1}^0$ a.s.\ with $\lambda+\delta<\beta$, then $U\in L^0(\Omega;C^{\lambda}(0,T;\tilde{E}_{a+\delta}))$.
\end{enumerate}
\end{theorem}

Note that because of the above result we can also view $L$ as a mapping from the subspace of adapted processes in $L^0(\Omega;C([0,T];\tilde{E}_a))$ into itself.

\begin{proof}
Choose $p\in (2, \infty)$ so large that $a+\delta + \lambda +\frac1p < \min\{\tfrac12 - \theta_B, 1-\theta_F,\eta_+\}$.

\textit{Existence}. First observe that $\|u_0\|_{E_{a,1}^0}$ is $\F_0$-measurable. Moreover, by Lemma \ref{lemma:interpol}, $t\mapsto S(t,0) u_0\in \tilde{E}_a$ has continuous paths.
Define $u_n = u_0\one_{\{\|u_0\|_{E_{a,1}^0}\leq n\}}$. Then $u_n$ is $\mathscr{F}_0$-measurable, and $t\mapsto S(t,0) u_n\in \tilde{E}_a$ has continuous paths and
\[\E\sup_{t\in [0,T]}\|S(t,0) u_n\|_{\tilde{E}_a}^p <\infty.\]
Hence by Theorem \ref{thm:existenceL^pcond}, problem \eqref{defn:SEE2} with initial condition $u_n$ admits a unique pathwise mild solution $U_n \in L^p(\Omega\times [0,T];\tilde{E}_a)$. Moreover, by Theorem \ref{thm:existenceL^pcond} and \eqref{eq:fractionalSobolevembReprForm} there exists a version of $U_n$ such that $U_n - S(t,0) u_n$ has paths in \[W^{\lambda+\frac1p,p}([0,T];\tilde{E}_{a+\delta}) \hookrightarrow C^{\lambda}([0,T];\tilde{E}_{a+\delta}).\]
In particular, $U_n$ has paths in $C([0,T];\tilde{E}_{a})$.
Moreover, almost surely on $\{\|u_0\|_{E_{a,1}^0}\leq m\}$, for all $t\in [0,T]$, $U_n(t) = U_m(t)$ when $n\geq m$. It follows that almost surely for all $t\in [0,T]$, the limit $\lim_{n\to\infty} U_n(t)$ exists in $\tilde{E}_a$. Define $U: \Omega\times [0,T] \to \tilde{E}_a$ by
\begin{align*}
U(t) = \begin{cases}
            \lim\limits_{n\to\infty} U_n(t) & \text{if the limit exists},\\
            0 & \text{else}.
        \end{cases}
\end{align*}
Then $U$ is strongly measurable and adapted. Moreover, almost surely on the set $\{\|u_0\|_{E_{a,1}^0} \leq n\}$, for all $t\in [0,T]$, $U(t) = U_n(t)$. Hence, almost surely, $U\in C([0,T];\tilde{E}_{a})$ and one can check that $U$ is a pathwise mild solution to \eqref{defn:SEE2}. By construction of $U$, there exists a version of $U - S(\cdot, 0)u_0$ with paths in $C^{\lambda}([0,T];\tilde{E}_{a+\delta})$ almost surely. In particular, $U$ has almost all paths in $C([0,T];\tilde{E}_a)$. If additionally $u_0\in E_{a+\beta,1}^0$ with $\delta+\lambda<\beta$, then by Lemma \ref{lemma:interpol}, $S(\cdot,0) u_0\in L^0(\Omega;C^{\lambda}(0,T;\tilde{E}_{a+\delta}))$.

\textit{Uniqueness}. Suppose $U^1$ and $U^2$ are both adapted pathwise mild solutions to \eqref{defn:SEE2} that belong to $L^0(\OO;C([0,T];\tilde{E}_s))$. We will show that almost surely $U^1 \equiv U^2$. For each $n\geq 1$ and $i=1,2$ define the stopping times
\[ \nu_n^i := \inf\Big\{ t\in [0,T]: \|U^i(t)\|_{\tilde{E}_a} \geq n\Big\},\]
where we let $\nu_n^i= T$ if the infimum is taken over the empty set.
Let $\tau_n = \nu^1_n \wedge \nu_n^2$ and $U^i_n = U^i \one_{[0,\tau_n]}$. Then $U^i_n \in Z^p_a$ and in a similar way as in Lemma \ref{lemma:as-equality} one can check that $\one_{[0,\tau_n]} U^i_n = \one_{[0,\tau_n]} L(U^i_n)$ for $i=1, 2$.
Therefore, from Lemma \ref{lemma:as-equality}, we find that almost surely for all $t\in [0,T]$, $U^1_n(t) = U^2_n(t)$. In particular, almost surely for almost all $t\leq \tau_n$, one has $U^1(t) = U^2(t)$. If we let $n\to\infty$ we obtain that almost surely, for all $t\in [0,T]$, one has $U^1(t) = U^2(t)$.
\end{proof}

\subsection{Results without uniformity conditions in $\Omega$\label{subs:generalO}}

In this section we will prove a well-posedness result for \eqref{defn:SEE2} without the uniformity condition \ref{hpths:unifrm-const}. The approach is based on a localization argument. Due to technical reasons we use a slightly different condition than (AT2), which is more restrictive in general, but satisfied in many examples. Details on this condition can be found in \cite{AT3} and \cite[Section IV.2]{Ama}. This condition is based on the assumption that $D(A(t))$ has constant interpolation spaces $E_{\nu,r} = (E_0, D(A(t)))_{\nu,r}$ for certain $\nu>0$ and $r\in [2, \infty)$, and the fact that the resolvent is $\mu$-H\"older continuous with values in $E_{\nu,r}$ with $\mu+\nu>1$. Note that in \cite[Section IV.2]{Ama} more general interpolation spaces are allowed. For convenience we only consider the case of constant real interpolation spaces.

\let\ALTERWERTA\theenumi
\let\ALTERWERTB\labelenumi
\def\theenumi{(CIS)}
\def\labelenumi{(CIS)}
\begin{enumerate}
\item\label{asmp:CIS-cndt}
Condition \ref{asmp:AT1} holds and there are constants $\nu\in (0,1]$ and $r\in [1, \infty]$ such that $E_{\nu,r}:= (E_0,D(A(t,\omega)))_{\nu,r}$ is constant in $t\in [0,T]$ and $\omega\in \Omega$ and there is a constant $C$ such that for all $x\in E_{\nu,r}$,
  \begin{align*}
            c^{-1}\|x\|_{E_{\nu,r}} &\leq \|x\|_{(E_0, D(A(t,\omega)))_{\nu,r}} \leq c\|x\|_{E_{\nu,r}}, \ \ t\in [0,T], \omega\in\Omega.
  \end{align*}
There is a $\mu\in (0,1]$ with $\mu+\nu>1$ and a mapping $K:\Omega\to \R_+$ such that for all $s,t\in [0,T]$, $\omega\in \Omega$,
\begin{align}\label{asmp:ATprime-cndt-estm}
\|A(t,\omega)^{-1} - A(s,\omega)^{-1}\|_{\calL(E_0,E_{\nu,r})} \leq K(\omega) (t-s)^{\mu}.
\end{align}
\end{enumerate}
\let\theenumi\ALTERWERTA
\let\labelenumi\ALTERWERTB
We have allowed $\nu = 1$ on purpose. In this way we include the important case where $D(A(t,\omega))$ is constant in time.

Clearly, this condition implies \ref{asmp:AT2} with constant $L(\omega)\leq C K(\omega)$. Indeed,
one has for all $\lambda\in \Sigma_{\vartheta}$:
\begin{equation}\label{eq:CISimpliesAT}
\begin{aligned}
\|&A(t,\omega)R(\lambda,A(t,\omega))(A(t,\omega)^{-1}-A(s,\omega)^{-1})\|_{\calL(E_0)} \\ & \leq
\|A(t,\omega)R(\lambda,A(t,\omega))\|_{\calL(E_{\nu,r},E_0)} \|A(t,\omega)^{-1}-A(s,\omega)^{-1}\|_{\calL(E_0,E_{\nu,r})}
\\ & \leq C K (t-s)^{\mu} \|R(\lambda,A(t,\omega))\|_{\calL(E_0,E_{1-\nu,r}^t)}
\\ & \leq C K  (t-s)^{\mu} |\lambda|^{-\nu}.
\end{aligned}
\end{equation}

We will now replace \ref{hpths:unifrm-const} by the following hypothesis.
\let\ALTERWERTA\theenumi
\let\ALTERWERTB\labelenumi
\def\theenumi{(H5)$'$}
\def\labelenumi{(H5)$'$}
\begin{enumerate}
  \item\label{asmp:CIS-cndtHH} Assume $E_0$ is separable.
        Assume $(A(t))_{t\in [0,T]}$ and $(A(t)^*)_{t\in [0,T]}$ satisfy \ref{asmp:CIS-cndt} with constants $\mu+\nu>1$ and $\mu^*+\nu^*>1$.
\end{enumerate}
\let\theenumi\ALTERWERTA
\let\labelenumi\ALTERWERTB
Unlike \ref{hpths:unifrm-const}, the mapping $K$ is allowed to be dependent on $\Omega$.

We can now prove the main result of this section which holds under the hypotheses \ref{asmp:A-mble-adptd}--\ref{hpths:interpol-spce-wkextra}, \ref{asmp:CIS-cndtHH}, \ref{asmp:Lip-asmpt-F} and \ref{asmp:Lip-asmpt-B}.
\begin{theorem}\label{thrm:exst-mild-sol:initialcondgeneral, generaltime, non-unfm-cnst}
Assume \ref{asmp:A-mble-adptd}--\ref{hpths:interpol-spce-wkextra}, \ref{asmp:CIS-cndtHH}, \ref{asmp:Lip-asmpt-F} and \ref{asmp:Lip-asmpt-B}.
Let $\delta, \lambda > 0$ be such that $a+\delta + \lambda < \min\{\tfrac12 - \theta_B, 1-\theta_F,\eta_+\}$.
Assume that $u_0:\Omega\to E_0$ is $\F_0$-measurable and $u_0\in E_{a,1}^0$ a.s.
Then assertions (1) and (2) of Theorem \ref{thrm:exst-mild-sol:initialcondgeneral, smalltime} hold.
\end{theorem}

Unlike in Theorem \ref{thm:existenceL^pcond} one cannot expect that the pathwise mild solution has any integrability properties in $\Omega$ in general. This is because of the lack of integrability properties of $S(t,s)$.

\begin{proof}\
For $\varepsilon \in (0,\mu)$ define $\phi:[0,T]\times\Omega\to \R_+$ by
\begin{align*}
\phi(t) &=
        \sup_{s\in [0,t)} \|A(t)^{-1} - A(s)^{-1}\|_{\calL(E_0,E_{\nu,r})} |t-s|^{-\mu+\varepsilon}, \ \  \text{if $t>0$},
\end{align*}
and $\phi(0) = 0$. Define $\phi^*$ in the same way for the adjoints $(A(t)^*)_{t\in [0,T]}$. It follows from \ref{asmp:CIS-cndtHH} and Lemma \ref{lemma:stp-time-cnt} that $\phi$ and $\phi^*$ are pathwise continuous. We claim that $\phi$ and $\phi^*$ are adapted. Since $E_0$ is separable, $\|A(t)^{-1} - A(s)^{-1}\|_{\calL(E_0,E_{\nu,r})}$ can be written as a supremum of countably many functions  $\|A(t)^{-1} x_n\|_{E_{\nu,r}}$, which are all $\F_t$-measurable by the Pettis measurability theorem. The claim follows.

Define the stopping times $\kappa_{n}, \kappa_n^*: \Omega\to\RR$ by $\kappa_{n} = \inf\{t\in [0,T]:\ \phi(t) \geq n\}$, $\kappa_n^* = \inf\{t\in[0,T]:\ \phi^*(t) \geq n\}$, and let $\tau_n = \kappa_n \wedge \kappa_n^*.$
Consider the stopped process $A_n$ given by $A_n(t,\omega) = A(t\wedge \tau_n(\omega), \omega).$
Then for all $s,t \in [0,T]$,
\[\|A_n(t)^{-1} - A_n(s)^{-1}\|_{\calL(E_0, E_{\nu,r})} \leq n|t-s|^{\mu-\varepsilon}\]
and similarly for $A_n(t)^*$, and it follows from \eqref{eq:CISimpliesAT} that $A_n$ and $A_n^*$ satisfy \ref{hpths:unifrm-const} with $\mu - \varepsilon$ instead of $\mu$, and with $L(\omega) = C n$. Let $(S_n(t,s))_{0\leq s\leq t\leq T}$ be the evolution family generated by $A_n$. Since $A_n(t) = A(t)$ for $t\leq \tau_n$, it follows from the uniqueness of the evolution family that $S_n(t,s)=S(t,s)$ for $0\leq s\leq t\leq \tau_n$.

\textit{Existence.}
Let the initial values $(u_n)_{n\geq 1}$ be as in the proof of
Theorem \ref{thrm:exst-mild-sol:initialcondgeneral, smalltime}.
It follows from Theorem \ref{thm:existenceL^pcond} that for each $n\geq 1$, there is a unique adapted pathwise mild solution $U_n\in L^p(\Omega;C([0,T];\tilde{E}_a))$ of \eqref{defn:SEE2} with $A$ and $u_0$ replaced by $A_n$ and $u_n$. Moreover, it also has the regularity properties stated in Theorem \ref{thrm:exst-mild-sol:initialcondgeneral, smalltime}. We will use the paths of $(U_n)_{n\geq 1}$ to build a new process $U$ which solves \eqref{defn:SEE2}.

For $v\in Z^p_a$ or $v\in L^0(\OO;C([0,T];\tilde{E}_a))$ we write
\begin{align*}
L(v)(t) &= S(t,0) u_0 + S*F(\cdot, v)(t) + S\diamond B(\cdot, v)(t),
\\   L_{n}(v)(t) &= S_n(t,0) u_n + S_n*F(\cdot, v)(t) + S_n\diamond B(\cdot, v)(t).
\end{align*}

Let $\Gamma_n = \{\|u_0\|_{E_{a}^0}\leq n\}$. Note that $L_n(U_n) = U_n$ for every $n$. Fix $m\geq 1$ and let $n\geq m$. Note that on $\Gamma_m$, $u_n = u_m$ and on $[0,\tau]$, $A_n$ = $A_m$.
By Lemma \ref{lemma:as-equality} we find that almost surely on the set $\Gamma_m$, if $t\leq \tau_m$, $U_n(t) = U_m(t)$. Therefore, we can define
\begin{align*}
U(t) = \begin{cases}
            \lim\limits_{n\to\infty} U_n(t) & \text{if the limit exists},\\
            0 & \text{else}.
        \end{cases}
\end{align*}
Then $U$ is strongly measurable and adapted.
Moreover, almost surely on $\Gamma_m$ and $t\leq \tau_m$, $U(t) = U_m(t)$. For $\omega\in \Omega$ and $m\geq 1$ large enough, $\tau_m(\omega)= T$. Thus the process $U$ has the same path properties as $U_m$, which yields the required regularity. One easily checks that $U$ is a pathwise mild solution to \eqref{defn:SEE2}.

\textit{Uniqueness.} Let $U^1$ and $U^2$  be adapted pathwise mild solutions in the space $L^0(\Omega;C([0,T];\tilde{E}_a))$. We will show that $U^1 = U^2$. Let $\kappa_n$ and $\kappa_n^*$ be as in the existence proof.
Let
\[\nu_n^i := \inf\{t\in [0,T]:\ \|U^i(t)\|_{\tilde{E}_a} \geq n\}, \qquad i=1,2.\]
Set $\nu_n = \kappa_n \wedge \kappa_n^* \wedge \nu_n^1 \wedge \nu_n^2$. Define $U_n^i$ by $U_n^i(t) = \one_{[0,\nu_n]}(t)U^i(t)$. Then as before one sees that $\one_{[0,\nu_n]} L_n(U_n^i) = U_n^i$.
Therefore, from Lemma \ref{lemma:as-equality}
it follows that almost surely, for all $t\in [0,\nu_n]$, $U^1_n = U^2_n$. The result follows by letting $n\to \infty$.
\end{proof}

\section{Examples\label{sec:ex}}
In this section, we will consider the stochastic partial differential equation from \cite{Veraar-SEE, S-SVui}. Let $(\OO,\F,\P)$ be a complete probability space with a filtration $(\F_t)_{t\in [0,T]}$.  For $p,q\in [1,\infty]$, $\alpha\in\RR$ and a domain $S\subseteq \RR^n$, let $B^s_{p,q}(S)$ denote the Besov space (see \cite{Tr1}).

\subsection{Second order equation on $\R^n$\label{sec:Rndomain}}

Consider the stochastic partial differential equation
\begin{equation}\label{eq:SPDEexample}
\begin{split}
du(t,s) &= \big(A(t,s,\omega, D)u(t,s) + f(t,s,u(t,s))\big)\,dt \\
 &\qquad + g(t,s,u(t,s))\, dW(t,s), \ t\in (0,T],\ s\in \RR^n, \\
u(0,s) &= u_0(s),\ s\in \RR^n.
\end{split}
\end{equation}
The drift operator $A$ is assumed to be of the form
\begin{align*}
A(t,s,\omega, D) = \sum_{i,j=1}^n D_i(a_{ij}(t,s,\omega)D_j) + a_0(t,s,\omega).
\end{align*}
We assume that all coefficients are real, and satisfy a.s.
\begin{align*}
\begin{split}
a_{ij} &\in C^{\mu}([0,T];C(\RR^n),\ a_{ij}(t,\cdot) \in BUC^1(\RR^n),\ D_ka_{ij} \in BUC([0,T] \times \RR^n),\\
a_0 &\in C^{\mu}([0,T];L^n(\RR^n)) \cap C([0,T];C(\RR^n)),
\end{split}
\end{align*}
for $i,j,k=1,\ldots, n$, $t\in [0,T]$ and a constant $\mu \in (\tfrac12, 1]$. All coefficients $a_{ij}$ and $a_0$ are $\mathscr{P}_T\otimes \mathscr{B}(S)$-measurable, where $\mathscr{P}_T$ is the progressive $\sigma$-algebra.
Moreover, there exists a constant $K$ such that for all $t\in [0,T]$, $\omega\in\OO$, $s\in \RR^n$, $i,j,k=1,\ldots,n$,
\begin{align*}
|a_{ij}(t,s,\omega)| \leq K, \ |D_ka_{ij}(t,s,\omega)| \leq K, \ |a_0(t,s,\omega)| \leq K.
\end{align*}
We assume there exists an increasing function $w: (0,\infty) \to (0,\infty)$ such that $\lim_{\e\downarrow 0} w(\e) = 0$ and such that for all $t\in [0,T]$, $\omega\in\OO$, $s,s'\in \RR^n$, $i,j=1,\ldots, n$,
\begin{align*}
|a_{ij}(t, s, \omega) - a_{ij}(t, s', \omega)| \leq w(|s-s'|).
\end{align*}

Moreover, we assume that $(a_{ij})$ is symmetric and that there exists a $\kappa >0$ such that
\begin{equation}\label{eq:exampleUEC}
\kappa^{-1} |\xi|^2 \leq a_{ij}(t,s,\omega)\xi_i\xi_j \leq \kappa |\xi|^2,\qquad s\in \RR^n,\ t\in [0,T],\ \xi\in \RR^n.
\end{equation}
Let $f,g: [0,T] \times \OO \times \RR^n \times \RR \to \RR$ be measurable, adapted and Lipschitz continuous functions with linear growth uniformly in $\OO\times [0,T] \times \RR^n$, i.e., there exist $L_f, C_f, L_g, C_g$ such that for all $t\in [0,T], \omega\in \OO, s\in \RR^n$ and $x,y\in \RR$,
\begin{align}
\label{eq:lipf} |f(t,\omega, s, x) - f(t,\omega, s, y)| &\leq L_f|x-y|, \\
\label{eq:lingrowthf}
|f(t,\omega, s, x)| &\leq C_f(1+|x|), \\
\label{eq:lipg} |g(t,\omega, s, x) - g(t,\omega, s, y)| &\leq L_g|x-y|, \\
\label{eq:lingrowthg}
|g(t,\omega, s, x)| &\leq C_g(1+|x|).
\end{align}
Let $W$ be an $L^2(\RR^n)$-valued Brownian motion with respect to $(\F_t)_{t\in [0,T]}$, with covariance $Q\in \calL(L^2(\RR^n))$ such that
\begin{equation}\label{ex:assumptioncovar}
\sqrt{Q} \in \calL(L^2(\RR^n), L^\infty(\RR^n)).
\end{equation}

Let $p \geq 2$ and set $E = L^p(\RR^n)$. On $E$, we define the linear operators $A(t,\omega)$ for $t\in [0,T]$, $\omega\in\OO$, by
\begin{align*}
D(A(t,\omega)) &= W^{2,p}(\RR^n), \\
A(t,\omega)u &= A(t,\cdot,\omega, D)u.
\end{align*}
By integration by parts, the adjoint $A(t,\omega)^*$ of $A(t,\omega)$ satisfies
\begin{align*}
D(A(t,\omega)^*) &= W^{2,p'}(\RR^n) \\
A(t,\omega)^*u &= A(t,\cdot,\omega, D)u.
\end{align*}
The operator $A(t,\omega): L^p(\RR^n) \to L^p(\RR^n)$ is a closed operator. In fact, from \cite[Theorem 8.1.1]{KryLecturesEllParab}, it follows that there exists a constant $c$ depending only on $p, \kappa, K, w$ and $n$, such that
\begin{align*}
c^{-1} \|x\|_{W^{2,p}(\RR^n)} \leq \|A(t,\omega)x\|_{L^p(\RR^n)} + \|x\|_{L^p(\RR^n)} \leq c\|x\|_{W^{2,p}(\RR^n)}, \ x\in W^{2,p}(\RR^n).
\end{align*}
By a careful check of the proof of \cite[Theorem 7.3.6]{Pazy}, it follows that one can find a sector $\Sigma_\vartheta$, $\vartheta \in (\pi/2, \pi)$, and a constant $M$, both independent of $t$ and $\omega$, such that for all $\lambda \in \Sigma_\vartheta$,
\begin{align*}
\|\lambda R(\lambda, A(t,\omega))\| \leq M.
\end{align*}
Changing $A(t,\omega)$ to $A(t,\omega) - \lambda_0$ and $f$ to $f+\lambda_0$ if necessary, it follows that \ref{asmp:AT1} holds. Note that the constant $C_f$ may be affected when replacing $f$ with $f+\lambda_0$, but it will remain independent of $t,\omega, s$ and $x$.
The operator $A(t,\omega)$ satisfies \ref{asmp:CIS-cndt} with $\nu=1$, see \cite[Theorem 4.1]{Yag90}. Hence \ref{hpths:ATplus} and \ref{asmp:CIS-cndtHH} are satisfied.

Hypothesis \ref{asmp:A-mble-adptd} holds by Example \ref{example:ConstDomH1Verify}.
To verify \ref{hpths:interpol-spce-wk}, take $\eta_+ = 1$ and for $\eta \in (0,\eta_+)$, set
\[ \tilde{E}_\eta := (L^p(\RR^n), W^{2,p}(\RR^n))_{\eta, p} = B^{2\eta}_{p,p}(\RR^n).\]
We do not need to choose an $\eta_-$, see Remark \ref{rmrk:InterpolSpaceHypoth}. Since $B^{2\eta}_{p,p}(\RR^n)$ has type $2$ and is a {\umd} space, \ref{hpths:interpol-spce-wkextra} holds.

Let $F: [0,T]\times\OO\times E \to E$ be defined by
\[F(t,\omega, x)(s) = f(t,\omega, s, x(s)).\]
Let $B:[0,T]\times\OO\times E\to \g(L^2(\RR^n), E)$ be defined by
\[(B(t,\omega, x)h)(s) = g(t,\omega, s, x(s))(\sqrt{Q}h)(s).\]
By assumption \eqref{ex:assumptioncovar} it follows that for any $x\in E$ and $h\in L^2(S)$,
\[|x(s) \cdot (Q h)(s)| \leq |x(s)| \|Q\|_{\calL(L^2(S),L^\infty(S))} \|h\|_{L^2(S)} \ \ \text{almost everywhere}.\]
Therefore, from \cite[Lemma 2.7]{Veraar-SEE} we can conclude that there is a constant $C$ depending on $Q$ and $p$ such that
\[\|x\sqrt{Q}\|_{\g(L^2(\RR^n), E)} \leq C\|x\|_{E}.\]
It follows that \ref{asmp:Lip-asmpt-F} and \ref{asmp:Lip-asmpt-B} are satisfied with choices $a=\theta_F=\theta_B = 0$. With the above definitions of $A$, $F$ and $B$, problem \eqref{eq:SPDEexample} can be rewritten as
\begin{equation}\label{ex:EvolEqExmpl1}
\begin{aligned}
du(t) &= (A(t)u(t) + F(t,u(t)))\;dt + B(t, u(t))\;dW(t), \\
u(0) &= u_0.
\end{aligned}
\end{equation}
Hence, if $\delta, \lambda > 0$ such that $\delta + \lambda < \tfrac12$, then Theorem \ref{thrm:exst-mild-sol:initialcondgeneral, generaltime, non-unfm-cnst} can be applied: there exists a unique adapted pathwise mild solution $u \in L^0(\OO;C([0,T];L^p(\RR^n)))$ to \eqref{ex:EvolEqExmpl1}. If additionally $u_0 \in W^{1,p}(\RR^n) = \tilde{E}_{1/2}$, then the solution $u$ belongs to the space $L^0(\OO;C^\lambda(0,T;B^{2\delta}_{p,2}(\RR^n)))$. This is summarized in the next theorem.
\begin{theorem}
Let $p\in (2,\infty)$ and suppose $u_0:\OO\to L^p(\RR^n)$ is $\mathscr{F}_0$-measurable.
\begin{enumerate}
\item There exists a unique adapted pathwise mild solution $u$ that belongs to the space $L^0(\OO;C([0,T];L^p(\RR^n)))$.
\item If $u_0 \in W^{1,p}(\RR^n)$ a.s., and $\delta, \lambda > 0$ such that $\delta +\lambda < \tfrac12$, then $u$ belongs to $L^0(\OO;C^\lambda(0,T;B^{2\delta}_{p,p}(\RR^n)))$.
\end{enumerate}
\end{theorem}

\begin{remark}\label{rem:improvementexamples} \
\begin{enumerate}
\item A first order differential term in problems \eqref{eq:SPDEexample} and \eqref{eq:SPDEexampleBddDom} may be included. This term may in fact be included in the function $f$. To handle such a situation, one needs to consider $a > 0$, $\theta_F > 0$.
\item One can also consider the case of non-trace class noise, e.g. space-time white noise, see for instance \cite{NVW-evolution}. In this situation one needs to take $a>0$, $\theta_B > 0$. Also in the case of boundary noise or random point masses, one can consider $a > 0, \theta_F > 0$ and $\theta_B > 0$, see \cite{SV-WaveEq, SV-BoundaryNoise}.
\end{enumerate}
\end{remark}

\subsection{Second order equation on a bounded domain $S$ with Neumann boundary conditions\label{sec:boundeddomain}}
Let $S$ be a bounded domain in $\RR^n$ with $C^2$-boundary and outer normal vector $n(s)$. Consider the equation
\begin{equation}\label{eq:SPDEexampleBddDom}
\begin{split}
du(t,s) &= \big(A(t,s,\omega, D)u(t,s) + f(t,s,u(t,s))\big)\,dt \\
 &\qquad + g(t,s,u(t,s))\, dW(t,s), \ t\in (0,T],\ s\in S, \\
C(t,s,\omega, D)u(t,s) &= 0, \ t\in (0,T],\ s \in \partial S, \\
u(0,s) &= u_0(s),\ s\in S.
\end{split}
\end{equation}
The drift operator $A$ is of the form
\begin{align}
A(t,s,\omega, D) &= \sum_{i,j=1}^n D_i(a_{ij}(t,s,\omega)D_j) + a_0(t,s,\omega), \label{eq:exampledrift}\\
C(t,s,\omega,D) &= \sum_{i,j=1}^n a_{ij}(t,s,\omega)n_i(s)D_j, \label{eq:ExmpBdCnt}
\end{align}
where $D_i$ stands for the derivative in the $i$-th coordinate.
A difficult in \eqref{eq:SPDEexampleBddDom} is that the boundary value operator changes with time.

All coefficients are real and satisfy a.s.
\begin{equation}\label{ex:assumptionscoefficients}
\begin{split}
a_{ij} &\in C^{\mu}([0,T];C(\overline{S})),\ a_{ij}(t,\cdot) \in C^1(\overline{S}),\ D_ka_{ij} \in C([0,T] \times \overline{S}),\\
a_0 &\in C^{\mu}([0,T];L^n(S)) \cap C([0,T];C(\overline{S})),
\end{split}
\end{equation}
for $i,j,k=1,\ldots, n$, $t\in [0,T]$ and a constant $\mu \in (\tfrac12, 1]$ (see Remark \ref{rem:mugroterhalf} for an improvement). Moreover, all other assumptions from Section \ref{sec:Rndomain} regarding $a_{ij}, a_0, f$ and $g$ \eqref{eq:lipf}, \eqref{eq:lingrowthf}, \eqref{eq:lipg}, \eqref{eq:lingrowthg} and
\begin{align*}
\sqrt{Q} \in \calL(L^2(S), L^\infty(S))
\end{align*}
are assumed to hold.

Let $p \geq 2$ and set $E = L^p(S)$. On $E$, we define the linear operators $A(t,\omega)$ for $t\in [0,T]$, $\omega\in\OO$, by
\begin{align*}
D(A(t,\omega)) &= \{u\in W^{2,p}(S):\ C(t,s,\omega,D)u=0,\ s\in \partial S\}, \\
A(t,\omega)u &= A(t,\cdot,\omega, D)u.
\end{align*}
With integration by parts, one observes that the adjoint $A(t,\omega)^*$ of $A(t,\omega)$ is given by
\begin{align*}
D(A(t,\omega)^*) &= \{u\in W^{2,p'}(S):\ C(t,s,D)u=0,\ s\in \partial S\}, \\
A(t,\omega)^*u &= A(t,\cdot,\omega, D)u.
\end{align*}

As in the previous example, $A(t,\omega)$ is a closed operator on $L^p(S)$, see \cite[Theorem 8.5.6]{KryLecturesEllParab} and the discussion in \cite[Section 9.3]{KryLecturesEllParab}. We have
\begin{align*}
c^{-1} \|x\|_{W^{2,p}(S)} \leq \|A(t,\omega)x\|_{L^p(S)} + \|x\|_{L^p(S)} \leq c\|x\|_{W^{2,p}(S)}, \ x\in D(A(t,\omega)).
\end{align*}
where the constant $c$ depends only on $p, \kappa, K, w, n$ and the shape of the domain $S$. As in the previous example, $A(t,\omega)$ and $A(t,\omega)^*$ both satisfy \ref{asmp:AT1}.

Next, we will show \ref{asmp:CIS-cndt}.
By \cite[Theorem 5.2]{Amann-nonh} and \cite[(5.25)]{Amann-nonh}, it follows that for $\nu < \frac12 + \frac{1}{2p}$,
\begin{align}\label{ex:intpoltimespaceindep}
(E, D(A(t,\omega)))_{\nu, p} = B^{2\nu}_{p,p}(S),
\end{align}
with constants independent of $t,\omega$. For $\nu < \frac12 + \frac{1}{2p'}$, we obtain the same result for the adjoint $A(t,\omega)^*$. Hence, for $\nu < \frac12$, by \eqref{ex:intpoltimespaceindep}, \eqref{array-cnt-emb1}, \eqref{array-cnt-emb2} and \cite[Theorem 4.1]{Yag90} we obtain for $\e>0$ such that $\nu+\e < \frac12$,
\begin{align*}
\|A(t)^{-1} - A(s)^{-1}\|_{\calL(E_0, B^{2\nu}_{p,q})} &\leq \|(-A(t))^{\nu+\e}(A(t)^{-1} - A(s)^{-1})\|_{\calL(E_0)} \\&\leq K(\omega)|t-s|^{\mu}.
\end{align*}
A similar estimation holds again for the adjoint. This proves \ref{asmp:CIS-cndt} and therefore \ref{asmp:CIS-cndtHH} and \ref{hpths:ATplus} (see also \eqref{eq:CISimpliesAT} and its discussion).

The verification of hypothesis \ref{asmp:A-mble-adptd} is given in the appendix, see Lemma \ref{lemma:(H1)-verification}.

To verify \ref{hpths:interpol-spce-wk}, take $\eta_+ = \tfrac12$ and $\tilde{E}_{\eta} := (E, W^{2,p})_{\eta, p}$. Note that in particular, regarding \eqref{ex:intpoltimespaceindep}, \ref{hpths:interpol-spce-wk}(ii) is satisfied. As in the previous example, we do not need to consider $\eta_-$. Also \ref{hpths:interpol-spce-wkextra} is satisfied by the choice of $\tilde{E}_\eta$.
The verification of (HF) and (HB) is as in Section \ref{sec:Rndomain}. In fact, we can take $a = \theta_F = \theta_B = 0$ again. This means that problem \eqref{eq:SPDEexampleBddDom} can be rewritten as a stochastic evolution equation
\begin{equation}\label{ex:EvolEqExmpl2}
\begin{aligned}
du(t) &= (A(t)u(t) + F(t,u(t)))\;dt + B(t, u(t))\;dW(t), \\
u(0) &= u_0.
\end{aligned}
\end{equation}
Hence, if $\delta, \lambda > 0$ such that $\delta + \lambda < \tfrac12$, then Theorem \ref{thrm:exst-mild-sol:initialcondgeneral, generaltime, non-unfm-cnst} can be applied: there exists a unique adapted pathwise mild solution $u\in L^0(\OO;C([0,T];L^p(S)))$ to \eqref{ex:EvolEqExmpl2}. Moreover, if $\beta < \tfrac12$ such that $\lambda+\delta < \beta$ and if $u_0 \in W^{1,p}(S)$ a.s., then $u \in L^0(\OO;C^{\lambda}(0,T;B^{2\delta}_{p,p}))$. Summarizing, we have the following result.
\begin{theorem}\label{Thm:ExistUniqueExmpl}
Let $p\in (2,\infty)$ and suppose $u_0:\OO\to L^p(S)$ is $\mathscr{F}_0$-measurable.
\begin{enumerate}
\item There exists a unique adapted pathwise mild solution $u$ that belongs to the space $L^0(\OO;C([0,T];L^p(S)))$.
\item If $u_0 \in W^{1,p}(S)$ a.s., and $\delta, \lambda > 0$ such that $\delta + \lambda < \tfrac12$, then $u$ belongs to $L^0(\OO;C^{\lambda}(0,T; B^{2\delta}_{p,p}(S))).$
\end{enumerate}
\end{theorem}

\begin{remark}\label{rem:mugroterhalf}
In the above we assumed $\mu>1/2$. However, it is clear that we can consider the case $\mu>\frac12 - \frac{1}{2p}$ as well. Moreover, the previous Remark \ref{rem:improvementexamples} applies in the above result as well.
\end{remark}

\appendix

\section{A technical result for H\"older continuous functions}

Let $X$ be a Banach space. For a given $\mu$-H\"older function $f:[0,T]\to X$ and $\alpha\in (0,\mu)$, let
\begin{align*}
\phi_{f,\alpha}(t) &=
        \begin{cases}
        \sup_{s\in [0,t)} \frac{\|f(t)- f(s)\|}{|t-s|^{\alpha}}, & \text{if $t\in(0,T]$}, \\
        0 & \text{if $t=0$}.
    \end{cases} \label{defn:fnct-stp-tme}
\end{align*}

\begin{lemma}\label{lemma:stp-time-cnt}
Let $f\in C^{\mu}([0,T];X)$ with $\mu\in (0,1]$. Then for every $\alpha \in (0,\mu)$, the function $\phi_{f,\alpha}$ is in $C^{\mu-\alpha}([0,T];X)$.
\end{lemma}

\begin{proof}
Let $C = [f]_{C^\mu([0,T];X)}$.
Let $\alpha\in (0,\mu)$ and write $\phi:=\phi_{f,\alpha}$, let $\varepsilon = \mu-\alpha$ and fix $0\leq \tau<t\leq T$.

Since $\phi$ is increasing we have $\phi(t) \geq \phi(\tau)$. If $\tau=0$, one can write $|\phi(t) - \phi(0)| \leq C\sup_{s\in [0,t)}(t-s)^{\varepsilon} = C t^{\varepsilon}$.
Next consider $\tau\neq 0$.

\textit{Step 1:} Assume that $\displaystyle \phi(t) = \sup_{s\in [\tau,t)} \|f(t) - f(s)\| (t-s)^{-\mu+\varepsilon}$. Then
\begin{align*}
|\phi(t) - \phi(\tau)| &\leq \phi(t) = \sup_{s\in [\tau,t)} \frac{\|f(t) - f(s)\|}{(t-s)^{\mu-\varepsilon}} \leq \sup_{s\in [\tau, t)} C(t-s)^\varepsilon = C (t-\tau)^\varepsilon.
\end{align*}

\textit{Step 2:} Now suppose $\displaystyle \phi(t) = \sup_{s\in [0,\tau)} \|f(t) - f(s)\| (t-s)^{-\mu+\varepsilon}$. Then one has
\[|\phi(t) - \phi(\tau)| \leq \sup_{s\in [0,\tau)} \Big| \frac{\|f(t) - f(s)\|}{(t-s)^{\mu-\varepsilon}} - \frac{\|f(\tau) - f(s)\|}{(\tau-s)^{\mu-\varepsilon}}\Big|.
\]
With the triangle inequality, we find that
\begin{align}
\nonumber |\phi&(t) - \phi(\tau)| \leq \\
&\sup_{s\in [0,\tau)} \frac{\|f(t) - f(\tau)\|}{(t-s)^{\mu-\varepsilon}} \nonumber
 + \sup_{s\in [0,\tau)} \|f(\tau) - f(s)\||(t-s)^{-\mu+\varepsilon} - (\tau-s)^{-\mu+\varepsilon}| \nonumber\\
  &\leq C \sup_{s\in [0,\tau)} \frac{(t-\tau)^{\mu}}{(t-s)^{\mu-\varepsilon}} + C \sup_{s\in [0,\tau)} (\tau-s)^{\mu}((\tau-s)^{-\mu+\varepsilon} - (t-s)^{-\mu+\varepsilon})
\nonumber\\
  &=C (t-\tau)^{\varepsilon} + C \sup_{s\in [0,\tau)} (\tau-s)^{\mu}((\tau-s)^{-\mu+\varepsilon} - (t-s)^{-\mu+\varepsilon})\label{estm:stp-time-left-cnt-estm2}.
\end{align}
We claim that for all $s\in [0,\tau)$,
\begin{equation}\label{eq:holderclaim}
(\tau-s)^{\mu}((\tau-s)^{-\mu+\varepsilon} - (t-s)^{-\mu+\varepsilon})\leq (t-\tau)^{\varepsilon}.
\end{equation}
In order to show this, let $u = \tau-s$ and $v = t-s$. Then $v-u = t-\tau$ and \eqref{eq:holderclaim} is equivalent to \[u^{\varepsilon} - v^{\varepsilon}\Big(\frac{u}{v}\Big)^{\mu}\leq (v-u)^{\varepsilon},  \ \  0<u<v\leq T.\]
Writing $u = x v$ with $x\in (0,1)$ and dividing by $v^\varepsilon$, the latter is equivalent to
\[{x^{\varepsilon} - x^{\mu}}\leq {(1-x)^{\varepsilon}},  \ \  x\in (0,1).\]
For all $x\in [0,1]$ one has ${x^{\varepsilon} - x^{\mu}}\leq 1-x^{\mu-\varepsilon}$. Thus it suffices to show that  ${1-x^{a}} \leq (1-x)^{b}$ where $a = \alpha\in (0,1)$ and $b=\varepsilon\in (0,1)$. However, ${1-x^{a}}\leq 1-x \leq (1-x)^b$ for all $x\in [0,1]$ and this proves the required estimate.

We can conclude that the right-hand side of \eqref{estm:stp-time-left-cnt-estm2} is less or equal than $2C(t-\tau)^{\varepsilon}$. This completes the proof.

\end{proof}

\section{Measurability of the resolvent}

\begin{lemma}\label{lemma:(H1)-verification}
The drift operator $A$ from Section \ref{sec:boundeddomain} satisfies condition (H1).
\end{lemma}
\begin{proof}\
We will prove adaptedness of the resolvent. Strong measurability can be done similarly, and will be omitted. Fix $t\in [0,T]$.

\textit{Step 1. Continuous dependence on the coefficients.}\\
Consider, besides the operator $A$, the operator $A'$ satisfying \eqref{eq:exampledrift} but with $a_{ij}'$ and $a_0'$ instead of $a_{ij}$ and $a_0$, respectively. We assume that $a_{ij}'$, $a_0'$ are functions satisfying \eqref{ex:assumptionscoefficients}. Consider the closed operators $A(t), A'(t): \OO\to\calL(L^p(S))$. Let $p'$ be the H\"{o}lder conjugate of $p$, let $f\in L^p(S)$ and $g\in L^{p'}(S)$. Set
\[u := (R(\lambda, A(t)) - R(\lambda, A'(t)))f \in W^{2,p}(S), \ \  v := R(\lambda, A(t)^*)g \in D(A(t)^*).\]
By applying \cite[(2.40)]{Schnaubelt} with $\nu = 0$ and $A'(t)$ instead of $A(s)$, we obtain
\begin{align*}
\lb (R&(\lambda, A(t)) - R(\lambda, A'(t)))f, g \rb = \int_S u (\lambda - A(t))v\;dx \\
&= \sum_{i,j=1}^n \int_S (a_{ij}'(t, x) - a_{ij}(t,x)) (D_jR(\lambda, A'(t))f)(x) (D_iR(\lambda, A(t)^*)g)(x)\;dx \\
&\qquad + \int_S (a_0'(t,x) - a_0(t,x)) (R(\lambda, A'(t))f)(x) (R(\lambda, A(t)^*)g)(x)\;dx.
\end{align*}
Still following the lines of \cite{Schnaubelt}, it follows that
\begin{align*}
|\lb (R&(\lambda, A(t)) - R(\lambda, A'(t)))f, g \rb| \\& \leq C \max_{i,j,\omega,x} \{|a_{ij}(t,x) - a_{ij}'(t,x)|, |a_0(t,x) - a_0'(t,x)| \} \|f\|_{L^p(S)} \|g\|_{L^{p'}(S)}.
\end{align*}
Hence
\begin{equation}\label{eq:RcontAAprime}
\|R(\lambda, A(t)) - R(\lambda, A'(t))\|_{\calL(L^p(S))} \leq C \max_{i,j,\omega,x} |a_{ij}(t,x) - a_{ij}'(t,x)|.
\end{equation}

\textit{Step 2. Approximation of the coefficients.}\\
Let us denote the space of all symmetric $n\times n$-matrices by $\RR^{n\times n}_{\textrm{sym}}$, endowed with the operator norm. Consider $a(t)$ as a map $a(t): \OO \to C^1(\overline{S}, \RR^{n\times n}_{\textrm{sym}})$. For $i,j = 1,\ldots, n$ and $s \in \overline{S}$, define $x_{i,j,s}^* \in C^1(\overline{S}, \RR^{n\times n}_{\textrm{sym}})^*$ by the point evaluation $\lb f, x_{i,j,s}^* \rb = f(s)_{ij}$. Let $\Gamma$ be the subset of $C^1(\overline{S}, \RR^{n\times n}_{\textrm{sym}})^*$ defined by
\[\Gamma := \{ x_{i,j,s}^* \in C^1(\overline{S}, \RR^{n\times n}_{\textrm{sym}})^*:\ i,j=1,\ldots, n,\ s\in \overline{S}\}.\]
Note that $\Gamma$ is a set separating the points of $C^1(\overline{S}, \RR^{n\times n}_{\textrm{sym}})$. Since for all $s\in \overline{S}$, $a_{ij}(t,s):\OO\to\RR$ is $\mathscr{F}_t$-measurable, by assumption, it follows from Pettis's theorem \cite[Proposition I.1.10]{VakTarCho} that $a(t)$ is $\mathscr{F}_t$-measurable. Hence, by \cite[Proposition I.1.9]{VakTarCho}, there exists a sequence of mappings $a^k(t): \OO \to C^1(\overline{S}, \RR^{n\times n}_{\textrm{sym}})$, such that $a^k(t)$ is countably valued and such that $a^k(t)^{-1}(f) \in \mathscr{F}_t$ for all $f\in C^1(\overline{S}, \RR^{n\times n}_{\textrm{sym}})$, with the property that $a^k(t) \to a(t)$ uniformly in $\OO$.
Let $\e >0 $ and choose $N\in \NN$ such that for all $k > N$ and all $\omega \in \OO$, $\sup_{s\in \overline{S}} \|a(t,s) - a^k(t,s)\|_{\RR^{n\times n}_s} < \e$. Since $a(t,s)$ is invertible, by the uniform ellipticity condition \eqref{eq:exampleUEC}, it follows that $a^k(t,s)$ is invertible whenever $k$ is large enough. In fact, by estimating the norm $\|a^k(t,s)^{-1}\|$, one obtains the following result: there exists a $\delta > 0$ and an $\tilde{N} \in \NN$ such that for all $k > N$, $a^k(t,s)$ satisfies \eqref{eq:exampleUEC} with a constant $\tilde{\kappa}$ such that $\tilde{\kappa} \in [\kappa, \kappa + \delta]$.

Consider the operator $A_k$ defined by \eqref{eq:exampledrift} but with $a^k_{ij}$ instead of $a_{ij}$. Note that $A_k$ satisfies \ref{asmp:AT1}. Since $a^k_{ij}$ is countably valued, $R(\lambda, A_k(t)): \OO \to \calL(L^p(S))$ is countably valued as well, and hence $\mathscr{F}_t$-measurable. By \eqref{eq:RcontAAprime}, we obtain $R(\lambda, A_k(t)) \to R(\lambda, A(t))$ as $k\to\infty$, uniformly in $\OO$, and therefore it follows that $R(\lambda, A(t))$ is $\mathscr{F}_t$-measurable.

To prove strong measurability, repeat step $1$ but with $A:\OO\times [0,T] \to \calL(L^p(S))$ instead of $A(t): \OO \to \calL(L^p(S))$. Similarly, in step $2$ one considers $a: \OO\times [0,T] \to C^1(\overline{S}, \RR^{n\times n}_{\textrm{sym}})$ and the $\sigma$-algebra $\mathscr{F} \otimes \mathscr{B}([0,T])$.

\end{proof}

\def\polhk#1{\setbox0=\hbox{#1}{\ooalign{\hidewidth
  \lower1.5ex\hbox{`}\hidewidth\crcr\unhbox0}}}

\end{document}